\documentclass[11pt]{article}
\usepackage{amsthm, amsfonts, amsxtra, amssymb, amscd}

\usepackage{amsfonts,amsmath,amssymb,amscd,amsthm}
\usepackage[english]{babel} 
\usepackage[all]{xy}
\usepackage[backref, colorlinks, linktocpage, citecolor = blue, linkcolor = blue]{hyperref}

\newtheorem*{lemma*}{Lemma}
\newtheorem{lemma}[subsection]{Lemma}
\newtheorem*{theorem*}{Theorem}
\newtheorem{theorem}[subsection]{Theorem}
\newtheorem*{proposition*}{Proposition}
\newtheorem{proposition}[subsection]{Proposition}
\newtheorem*{corollary*}{Corollary}

\newtheorem{corollary}[subsection]{Corollary}

\theoremstyle{definition}
\newtheorem*{definition*}{Definition}
\newtheorem{definition}[subsection]{Definition}
\newtheorem*{example*}{Example}
\newtheorem{example}[subsection]{Example}

\theoremstyle{remark}
\newtheorem*{remark*}{Remark}
\newtheorem{remark}[subsection]{Remark}

\renewcommand{\phi}{\varphi}

\newcommand{\be}{\begin{enumerate}}
\newcommand{\ee}{\end{enumerate}}

\title{Norms and Cayley Hamilton algebras}
\author{Claudio procesi}
\begin{document}\begin{abstract} We develop the general Theory of    Cayley Hamilton algebras  using norms and compare with the approach, valid only in characteristic 0,  using traces and presented in a previous paper \cite{ppr}.
\end{abstract}
\maketitle
\hfill{\it\small To the memory of  Edoardo Vesentini}
\tableofcontents
\section*{Foreword}  A basic fact for an $n\times n$   matrix $a$ (with entries in a commutative ring)  is the construction of its characteristic polynomial $\chi_a(t):=\det(t-a)$, $t$ a variable,  and the Cayley Hamilton theorem $\chi_a(a)=0$.

The notion  of  Cayley Hamilton algebra (CH algebras for short), see Definition \ref{CHalg},  was introduced in 1987 by Procesi \cite{P5} as an axiomatic treatment  of the Cayley Hamilton theorem. This was done in order to clarify the Theory of $n$--dimensional representations, cf. Definition \ref{ndi}, of an associative and in general noncommutative algebra $R$. With 1, unless otherwise specified, (from now on just called {\em algebra}). 

The theory was developed only  in characteristic 0, for two reasons;  the first being that at that time it was not clear to the author if the characteristic free results of Donkin \cite{Don} and Zubkov \cite{Zub}  were sufficient to found the theory in general. The second reason was mostly because it looked  not likely that the {\em main theorem} \ref{immu}  could possibly hold in general.

The first concern can now be considered to have a positive solution, due to the contributions of several people, and we may take the book \cite{depr0} as reference. As for the second, that is the main theorem in positive characteristic, the issue remains unsettled. The present author feels that it should not be true in general but has no counterexamples.\medskip

A  partial theory in general characteristics replacing the trace with the determinant appears already in Procesi \cite{P2} and  \cite{P6}.

  For the general definition, see \ref{CHAm} for details,  we   follow Chenevier \cite{Che}:
\begin{definition}\label{ChAl}
Given an algebra $R$ over a commutative ring $A$ and $n\in\mathbb N$, an {\em $n$--norm} is a multiplicative polynomial law, as $A$--modules,  $N:R\to A$  homogeneous of degree $n$ (see Definition \ref{pola}).

An algebra $R$ over a commutative ring $A$  with an $n$--norm  $N:R\to A$ is a Cayley Hamilton algebra if, for every commutative $A$ algebra $B$, each $a\in B\otimes_AR$ satisfies its  {\em characteristic polynomial}  $\chi_a(t):=N(t-a)$, that is  $\chi_a(a)=0.$
\end{definition} The original definition over $\mathbb Q$  is through the axiomatization of a {\em trace},   and closer to the Theory of {\em pseudocharacters} see \cite{ppr}.  The definition via trace is also closer to the languange un Universal algebra, while the one using norms is more categorical in nature.

This                 paper is a continuation of   \cite{ppr} where we have developed the theory in characteristic 0  using the notion of trace algebra. Here instead the approach is characteristic free and through the axiomatization of Norms.\smallskip

The first   section  of this paper forms an exposition of known results  with one or two new facts or proofs.    We suggest the reader to start at \S \ref{scha}: $n$--Cayley--Hamilton  algebras.
In \S \ref{leCH} we treat  the general structure theory of Cayley--Hamilton algebras and some structure of  their $T$--ideals.
\section{Invariants and representations}
\subsection{$n$--dimensional representations} Let us recall some basic facts which are treated in detail in the forthcoming book with Aljadeff, Giambruno and Regev \cite{agpr}.\smallskip

For a given $n\in\mathbb N$ and a ring $A$  by $M_n(A)$ we denote the ring of $n\times n$ matrices with coefficients in $A$, by a symbol $(a_{i,j})$  we denote a matrix with entries $ a_{i,j}\in A,\ i,j=1,\ldots,n $.

In particular we will usually assume $A$  commutative so that the construction $A\mapsto M_n(A)$ is a functor from the category $\mathcal C$ of commutative rings to that $\mathcal R$ of associative rings. To a map $f:A\to B$ is associated a  map $M_n(f):M_n(A)\to M_n(B)$  in the obvious way $M_n(f)((a_{i,j})):= (f(a_{i,j}))$.
\begin{definition}\label{ndi}
By an  {\em $n$--dimensional representation}  of a ring $R$ we mean a homomorphism $f:R\to M_n(A)$ with $A$ commutative.
\end{definition}
The set valued functor $A\mapsto \hom_{\mathcal R}(R, M_n(A))$ is representable. That is, there is a commutative ring $T_n(R)$  and a natural isomorphism $j_A$
$$ \hom_{\mathcal R}(R, M_n(A))\stackrel{j_A}\simeq  \hom_{\mathcal C}(T_n(R),  A  ),\quad j_A:f\mapsto \bar f$$ given by the universal map $\mathtt j_R:R\to M_n(T_n(R))$  and a commutative diagram $f=M_n(\bar f)\circ \mathtt j_R$:
\begin{equation}\label{unpps}
\xymatrix{  R\ar@{->}[r]^{\mathtt j_R\qquad}\ar@{->}[rd]_f&M_n(T_n(R))\ar@{->}[d]^{M_n(\bar f)}\\
&M_n(A) }\quad  .
\end{equation} The map $\mathtt j_R:R\to M_n(T_n(R))$ is called the {\em universal $n$--dimensional representation of $R$ } or {  \em the universal map into $n\times n$ matrices}.

Of course it is possible that $R$ has no   $n$--dimensional representations, in which case $T_n(R)=\{0\}$. \medskip

Of course the same discussion can be performed when $R$  is in the category $\mathcal R_F$ of algebras over a commutative ring $F.$ 
In this case  the functor   $A\mapsto \hom_{\mathcal R_F}(R, M_n(A))$ is on commutative $F$ algebras and $T_n(R)$ is an $F$ algebra.\smallskip

The construction of $\mathtt j_R$ is in two steps. First one easily sees that when $R=F\langle x_i\rangle_{i\in I}$ is a free algebra  then:
\begin{proposition}\label{genma}
$T_n(R)=F[\xi_{h,k}^{(i)}]$  is the polynomial algebra over $F$ in the variables $\xi_{h,k}^{(i)},\ i\in I,\ h,k=1,\ldots, n$ and $\mathtt j_R(x_i)=\xi_i:=(\xi_{h,k}^{(i)})$ the {\em generic matrix} with entries $\xi_{h,k}^{(i)}$.
\end{proposition}

For a general algebra $R$  one may present it as a quotient $R=F\langle x_i\rangle/I$ of a free algebra. Then $\mathtt j_{F\langle x_i\rangle}(I)$ generates in $M_n( F[\xi_{h,k}^{(i)}])  $ an ideal which is, as any ideal in a matrix algebra, of the form  $M_n(J),$ with $J$ an ideal of $F[\xi_{h,k}^{(i)}]$. Then the universal map for $R$ is given by  $$\begin{CD} F\langle x_i\rangle@>>> M_n(F[\xi_{h,k}^{(i)}] ) \\@VVV@VVV\\
R@>\mathtt j_R>> M_n(F[\xi_{h,k}^{(i)}]/J).
\end{CD}$$  By the universal property this is   independent of the presentation of $R$.

\subsection{Generic matrices and invariants}
\begin{definition}\label{genmad}
The subalgebra  $ F\langle \xi_i|_{i\in I}\rangle$  of  $M_n( F[\xi_{h,k}^{(i)}]),$ $ i\in I,\ h,k=1,\ldots, n  $ generated by the matrices $\xi_i$  is called the algebra of  {\em generic matrices.}
\end{definition}

A classical Theorem of Amitsur  states that, if $F$ is a domain then  $ F\langle \xi_i|_{i\in I}\rangle$ is a domain. If $I$ has $\ell$ elements we also denote $F\langle \xi_i|{i\in I}\rangle= F_n (\ell) $. If $\ell\geq  2$ then   $F_n (\ell) $ has a division ring of quotients $D_n(\ell)$  which is of dimension $n^2$ over its center $Z_n(\ell)$. These algebras have been extensively studied.    One defines first  the commutative subalgebra  $T_n(\ell)\subset  Z_n(\ell)$ generated by the coefficients $\sigma_i(a)$ of the characteristic polynomial  $\det(t-a)=t^n+\sum_{i=1}^n(-1)^i\sigma_i(a)t^{n-i},$ $ \forall a\in F_n(\ell)$.    Next define $\mathcal S_n(\ell)=F_n(\ell)T_n(\ell)\subset D_n(\ell),$    one can understand $\mathcal S_n(\ell)$ and $T_n(\ell)$  by invariant theory, see Remark \ref{sftg}.\label{prSl}

The invariant theory involved is presented in \cite{P3} when $F$ is a field of characteristic 0, and may be considered as the {\em first and second fundamental theorem} of matrix invariants. For a characteristic free treatment see the book \cite{depr0}. In general, for simplicity of exposition  assume that $F$ is an infinite field:
\begin{theorem}\label{FFT}
The algebra $T_n(\ell)$ is the algebra of polynomial invariants  under the simultaneous action of $GL(n,F)$ by conjugation on the space $M_n(F)^\ell$ of $\ell$--tuples of $n\times n$ matrices.\smallskip

The algebra $\mathcal S_n(\ell)$ is the algebra of $GL(n,F)$--equivariant polynomial maps  from  the space $M_n(F)^\ell$ of $\ell$--tuples of $n\times n$ matrices to $M_nF)$.

\end{theorem}
As usual together with a   first fundamental theorem one may ask for a  {\em second fundamental theorem} which was proved independently by Procesi \cite{P3} and Razmyslov \cite{Raz3}  when $F$ has characteristic 0 and by Zubkov \cite{Zub} in general.  In this paper it will appear as characterization of free Cayley--Hamilton algebras, Theorem \ref{DoZu}.
 \smallskip

For a general algebra, quotient of the free algebra
again one may add to $R$  the algebra $T_n(R)$ generated    by the coefficients of the characteristic polynomial  $\sigma_i(a),\ \forall a\in j_R(R)$.  

\subsection{Symmetry}  The functor $\hom_{\mathcal R}(R, M_n(A))$  has a group of symmetries: the {\em projective linear group $PGL(n)$}.

It is best to define this as a representable group valued functor  on the category $\mathcal C$ of commutative rings. The functor associates to a commutative ring $A$ the group $\mathfrak G_n(A):=Aut_A(M_n(A))$ of $A$--linear automorphisms of the matrix algebra $M_n(A)$. 
One has a natural homomorphism of the general linear group $GL(n,A)$ to $\mathfrak G_n(A)$ which associates to an invertible matrix $X$  the inner automorphism  $a\mapsto XaX^{-1}$.

The functor {\em general linear group $GL(n,A)$} is represented by   the  Hopf algebra   $\mathbb Z[x_{i,j}][d^{-1}],\ i,j=1,\ldots,n$   with $d=\det(X),\ X:=(x_{i,j})$  with the usual  structure given compactly by comultiplication $\delta$, antipode $S$ and counit $\epsilon$:
$$\delta(X)=X\otimes X,\ S(X):=X^{-1},\ \epsilon:X\to 1_n.$$

The functor $\mathfrak G_n(A)$ is represented by  the sub  Hopf algebra, of $GL(n,A)$, $P_n\subset \mathbb Z[x_{i,j}][d^{-1}]$  formed  of elements homogeneous of degree 0. It has a basis, over $\mathbb Z$, of elements $ad^{-h}$  where $a$  is a doubly standard tableaux with no rows of length $n$  and of degree $h \cdot n$. For a proof see \cite{agpr} Theorem 3.4.21.

\begin{remark}\label{chab}
Of course if we work in the category of commutative $F$ algebras we replace $P_n$ with     $P_n\otimes_{\mathbb Z}F\subset F[x_{i,j}][d^{-1}] $.             
\end{remark}
The identity map $1_{P_n}:P_n\to P_n$  induces a {\em generic automorphism of $M_n(P_n)$}  which can be given  by {\em conjugation} by the generic matrix $X$:
 \begin{equation}\label{cogm}
J:M_n(P_n)\to M_n(P_n),\quad J:a\mapsto XaX^{-1}\in M_n(P_n).
\end{equation}  Observe that $X$ is not  a matrix in $P_n$ only the entries of $XaX^{-1}$ are in $P_n$ i.e. are homogeneous of degree 0 in the entries of $X$. If we denote by  $y_{i,j}$ the $i,j$ entry  of $X^{-1}$ (given by the usual formula as the cofactor divided by the determinant), we have 
$$Xe_{i,j}X^{-1}= (z_{h,k}^{(i,j)}),\quad z_{h,k}^{(i,j)}=x_{h,i}y_{j,k}.$$ In fact the entries  $z_{h,k}^{(i,j)}=x_{h,i}y_{j,k} $ of the matrices $Xe_{i,j}X^{-1}$ generate $P_n$ as algebra and the  ideal of relations is generated by the relations expressing the fact that  the map $e_{i,j}\mapsto  (z_{h,k}^{(i,j)})$ is a homomorphism (and then automatically an isomorphism).\smallskip

Then given an automorphism $g\in Aut_A(M_n(A))$  its associated classifying map $\bar g:P_n\to A$ fits in the  commutative diagram:
\begin{equation}\label{unpp1}
\xymatrix{  M_n(P_n)\ar@{->}[r]^{J}\ar@{->}[d]_{M_n(\bar g)}&M_n(P_n)\ar@{->}[d]^{M_n(\bar g)}\\
M_n(A)\ar@{->}[r]^g&M_n(A) } 
\end{equation}
 Finally we have an action of $\mathfrak G_n(A)$ on $\hom_{\mathcal R}(R, M_n(A))$  by composing a map $f$  with an automorphism $g$. One has a commutative diagram
\begin{equation}\label{unpp}
\xymatrix{  R\ar@{->}[r]^{\mathtt j_R\qquad}\ar@{->}[d]_f&M_n(T_n(R))\ar@{->}[d]^{M_n(\overline{g\circ f)}}\\
M_n(A)\ar@{->}[r]^g&M_n(A) } .
\end{equation}  Assume now that $R$  is an $F$ algebra so also $T_n(R)$ is an $F$ algebra and    $M_n(T_n(R))=T_n(R)\otimes_F M_n(F)$. An  automorphism $g$ of $M_n(F)$ induces  an  automorphism $1\otimes   g$ of $M_n(T_n(R))$. Take  $A=T_n(R)$ and $f= \mathtt  j_R$ in Formula \eqref{unpp} and  set $\hat g:=\overline{1\otimes g\circ \mathtt  j_R}$, an automorphism of $T_n(R)$,  so that $ 1\otimes g\circ \mathtt  j_R=M_n(\hat g)\circ \mathtt  j_R$. If $g_1,g_2$ are two automorphisms of $M_n(F)$ we have:
$$M_n(   \widehat {g_1\circ   g_2} ) \circ\mathtt  j_R= 1\otimes (g_1\circ g_2)\circ \mathtt  j_R=1\otimes g_1 \circ  M_n( \hat g_2 ) \circ\mathtt  j_R  $$ $$= M_n( \hat g_2 ) \circ 1\otimes  g_1\circ \mathtt  j_R=  M_n( \hat g_2 )  \circ M_n( \hat g_1 ) \circ \mathtt  j_R$$implies $ \widehat {g_1\circ   g_2}  = \hat g_2  \circ  \hat g_1 $.
Which  implies that the map $g\mapsto  \hat g$ is an antihomomorphism from $Aut_{F}(M_n(F))$ to 
 $Aut(T_n(R)).$ Finally
\begin{proposition}\label{inva} The map  $g\mapsto g\otimes \hat g^{-1} $ is a homomorphism from the group  $Aut_FM_n(F)$ to the group of  all automorphisms of $M_n(T_n(R))$. The image of $R$ under $j_R$ is formed of invariant elements. 

\end{proposition} 

 This is particularly simple when $F$ is an infinite field. In this case the group $Aut_FM_n(F)=PGL(n,F)$ is Zariski dense in $PGL(n,\bar F)$ with $\bar F $ an algebraic closure of $F$. Otherwise this can be set in the language of polynomial laws. 
 \begin{proposition}\label{inva1}   For every commutative $F$ algebra $B$,  an automorphism $g\in Aut_BM_n(B)$  induces an automorphism of the $B$ algebra $T_n(R)\otimes_FB$. The map  $g\mapsto g\otimes \hat g^{-1} $ is a homomorphism from the group  $Aut_BM_n(B)$ to the group of  all automorphisms of $M_n( B)\otimes_B(T_n(R)\otimes_FB)=M_n(T_n(R)\otimes_FB)$. The image of $_BR$ under $j_R\otimes 1_B$ is formed of invariant elements. 

\end{proposition}  The functor $\mathfrak G_n(A)\times \hom_{\mathcal R}(R, M_n(A)$ is represented by $ P_n\otimes T_n(R)$. Given  $f\in \hom_{\mathcal R}(R, M_n(A))$ associated to $\bar f:T_n(R)\to A$ and an automorphism $g$  of $M_n(A)$  associated to $\bar g:P_n\to A$ the pair   is associated to 
$$  P_n\otimes T_n(R) \stackrel{\bar g\otimes \bar f  }\longrightarrow A\otimes A  \stackrel{m }\longrightarrow A,\quad m(a\otimes b):=ab. $$
  The natural transformation $\mathfrak G_n(A)\times \hom_{\mathcal R}(R, M_n(A))\to  \hom_{\mathcal R}(R, M_n(A))$ induces the coaction on the classifying rings $\eta:T_n(R)\to P_n\otimes T_n(R)$.

Thus the composition $g\circ f\in \hom_{\mathcal R}(R, M_n(A))$ is associated to 
$$\overline{g\circ f}: T_n(R)\stackrel{\eta }\longrightarrow P_n\otimes T_n(R) \stackrel{\bar g\otimes \bar f  }\longrightarrow A\otimes A  \stackrel{m }\longrightarrow A,\quad m(a\otimes b):=ab. $$

In particular for $R= \mathbb Z\langle x_i\rangle_{ i\in I} $  we have  $T_n( \mathbb Z\langle x_i\rangle_{ i\in I}) $ is the polynomial ring in the entries  $\xi_{i,j}^{(k)}$ of the generic matrices $\xi_k$  and the action is via the formula
$$\xi_k\mapsto X\xi_kX^{-1},\ \xi_{i,j}^{(k)} \mapsto\sum_{a,b}x_{i,a}\xi_{a,b} ^{(k)}y_{b,j},\ X^{-1}=(y_{i.j}).$$ We have $x_{i,a} y_{b,j}\in P_n$. The induced map  $T_n( \mathbb Z\langle x_i\rangle_{ i\in I})\to  P_n\otimes T_n(F\langle x_i\rangle_{ i\in I})$ can be identified to the coaction.  
\begin{equation}\label{coac}
\eta:T_n(R)\to P_n\otimes T_n(R),\quad\eta(\xi_{i,j}^{(k)})=\sum_{a,b=1}^nx_{i,a}\xi_{a,b} ^{(k)}y_{b,j}.
\end{equation}  
 \begin{remark}\label{abin} In general the invariants are the elements invariant under the {\em generic automorphism}. In our case:
\begin{equation}\label{finv}
T_n(R)^{\mathfrak G_n}:=\{a\in T_n(R)\mid \eta(a)=1\otimes a\}.
\end{equation}
\end{remark}

\section{$n$--Cayley--Hamilton  algebras\label{scha}} 
\subsection{Polynomial laws and determinants} Given a commutative ring $F$,  an $F$ module $M$, and a commutative $F$ algebra $B$  one has the {\em base change} functor  from $F$--modules to $B$--modules:
\begin{equation}\label{bac}
_BM:=B\otimes_FM.
\end{equation} Recall that, in  \cite{Roby} and  \cite{Roby1}, Roby  defines:
\begin{definition}\label{pola}
A {\em polynomial law} between two $F$  modules $M,N$  is a natural transformation  of the two   set valued functors on the category $\mathcal C_F$ of commutative $F$ algebras:
\begin{equation}\label{polaw}
f_B:_B\!M\to _B\!N,\quad B\in \mathcal C_F.
\end{equation}             
\end{definition} 
Such a law is {\em homogeneous of degree $n$} if:
$$f_B(b a)=b^nf_B(a),\ \forall b\in B,\  \forall a\in _B\!M,\ \forall B\in \mathcal C_F.$$
Given $M$  let us denote, for $N$ any   module, by $\mathcal P_n(M,N)$  the  polynomial laws homogeneous of degree $n$ from $M$ to $N$. This, by Roby's Theory,  is a set valued functor  on     modules $N$ and 
 it is representable. 

This is done by constructing  the   divided powers $\Gamma_n(M)=\Gamma_{n,F}(M)$  (over the base $F$) together with the map $i_M:m\mapsto m^{[n]}$ of $M$   to  $\Gamma_n(M)$. \smallskip

The   divided powers $\Gamma_n(M)$ are constructed by generators and relations. The construction is compatible with base change, that is given a commutative $F$ algebra $B$ we have:
$$\Gamma_{n,B}(_BM) = _B\!\Gamma_{n,F}(M).$$

In fact in most applications there is a more concrete description of  $\Gamma_n(M)$. For instance if $M$ is a free (or just projective) $F$ module
one describes the divided power as symmetric  tensors: \begin{equation}\label{simten}
\Gamma_n(M)\simeq (M^{\otimes n}) ^{S_n},\quad M^{\otimes n}=M\otimes_FM\otimes_FM\otimes\cdots\otimes_FM.
\end{equation} From now on we assume to be in this case. Notice that by change of coefficient rings if $F\to A$ is a map of commutative rings we have $ _AM:=A\otimes_FM$ and  
$$ _AM^{\otimes n}=A\otimes_FM^{\otimes n},\quad _A\Gamma_n(M)=\Gamma_n(_AM). $$

\begin{remark}\label{fpf}
In general a polynomial law  is not determined by the map $f:M\to N$.  Under further hypothesis this is the case, the simplest being that the base commutative ring $F$ contains an infinite field.
\end{remark}
\subsubsection{Multiplicative polynomial laws}If we have two $F$--algebras $R,S$  we have the notion of {\em multiplicative polynomial law  $d:R\to S$}  that is 
$$ d(ab)=d(a)d(b),\ \forall a,b\in _B\!\!R,\ \forall B.$$

One  proves that, if $R$  is an algebra, then  $\Gamma_n(R)$ is also an algebra, which we call {\em the $n^{th}$--Schur algebra of $R$}, see \cite{depr0},  
 and $i_R$ is a universal  multiplicative polynomial map, homogeneous of degree $n$. That is any multiplicative polynomial law, homogeneous of degree $n$ from $R$ to an algebra $S$  factors through $i_R$ and a homomorphism of $\Gamma_n(R)$ to $ S $.   

Denoting by $\mathcal M_n(R,S)$  the set of   multiplicative polynomial laws homogeneous of degree $n$ from $R$ to $S$ we have the isomorphism:
\begin{equation}\label{mulre} \phi: d\mapsto d\circ i_R,\quad 
\hom_{\mathcal R}(\Gamma_n(R),S)\stackrel{\phi}\simeq \mathcal M_n(R,S),\quad S\in\mathcal R.
\end{equation} When  we have $\Gamma_n(R) =(R^{\otimes n})^{S_n}$, Formula \eqref{simten},
the algebra structure on $\Gamma_n(R)$ is  induced by the tensor product of algebras.     
 
 If $S$ is any algebra we denote by $S_{ab}$  its {\em abelianization}, that is $S$ modulo the ideal $C_S$ generated by all commutators $[x,y]=x\cdot y-y\cdot x$.   
     
When we restrict, in Formula \eqref{mulre}, to $S$  commutative, this functor, on commutative algebras, is    represented by  $\pi:\Gamma_n(R)\to \Gamma_n(R)_{ab}$  that is the abelianization  of $\Gamma_n(R)$.
\begin{equation}\label{mulrea}
\mathcal M_n(R,A)\simeq \hom_{\mathcal C}(\Gamma_n(R)_{ab},A),\quad A\in\mathcal C.
\end{equation} 
This is an object studied  by Roby in \cite{Roby1}, and by Ziplies in \cite{Ziplies1} and discussed in some detail in my book with De Concini \cite{depr0}. 
 
\subsection{$n$--Cayley--Hamilton  algebras}
Consider an algebra $R$,  over a commutative ring $Z$.
\begin{definition}\label{lanorm}
A multiplicative polynomial law, of $Z$ algebras, homogeneous  of degree $n$,  $N:R\to Z$  wlll be called a {\em norm}. 
\end{definition}  One may apply Roby's theory. First $N(r)=\bar N(r^{\otimes n})$ then, by Formula \eqref{mulrea},
   the map $\bar N$ factors through the abelianization $\Gamma_n(R)_{ab}$ of $\Gamma_n(R)$ (we still denote it by $\bar N$). \begin{equation}\label{unpp2}
\xymatrix{  R\ar@{->}[r]^{r\mapsto r^{\otimes n}}\ar@{->}[rd]_{N}&\Gamma_n(R) \ar@{->}[d]^{ \bar N}\ar@{->}[r]^\pi&\Gamma_n(R)_{ab}\ar@{->}[d]^{ \bar N}\\
&Z\ar@{->}[r]^1&Z}  \qquad  \begin{matrix}
\\\\\Gamma_n(R) \ \text{the}\ n^{th} \ \text{divided power}. 
\end{matrix} 
\end{equation}

  For $a\in B\otimes_ZR$ and a commutative variable $t$ we have $t-a\in  B[t]\otimes_ZR$ and can define the characteristic polynomial        of $a$ as $\chi_a(t):=N(t-a)\in B[t]$.

\begin{definition}\label{CHAm}
 We say that $R$ is an $n$--Cayley--Hamilton  algebra if one has the analogue of the Cayley--Hamil\-ton Theorem $\chi_a(a)=0,\ \forall a\in _B\!R,\ \forall B$.

\end{definition}  Sometimes we will use the short notation CH--algebra for Cayley--Hamil\-ton  algebra.
\begin{remark}\label{vl} In \cite{Pr8} we treated the Theory in characteristic 0. In this case it is better to use instead of the norm the trace. This at the same time simplifies the treatment but also yields    stronger results due to the linear reductivity of the linear group. 

\end{remark}

 We start with an important example.\medskip

If $F$ is an infinite field  and $V$ a vector space of some finite dimension $m$ over $F$  one has for $R=End_F(V)$ that
$$(End_F(V)^{\otimes n})^{S_n} = End_F(V ^{\otimes n} )^{S_n} = End_{F[S_n]}(V ^{\otimes n}) .$$
The algebra $End_{F[S_n]}(V ^{\otimes n})$ is spanned by the elements $g^{\otimes n}$ where $g\in GL(V)$.

If the characteristic of $F$ is 0 one has  the decomposition 
$$V ^{\otimes n}=\oplus_{\lambda\vdash n\mid ht(\lambda)\leq m} M_\lambda\otimes S_\lambda(V)$$  where  $M_\lambda$ is the irreducible representation  of  $S_n$  corresponding to $\lambda$ and $S_\lambda(V)$, a Schur functor, the  corresponding  irreducible representation  of  $GL(V)$. One proves that
\begin{equation}\label{swd}
End_F(V ^{\otimes n} )^{S_n}=\oplus_{\lambda\vdash n\mid ht(\lambda)\leq m}   End( S_\lambda(V) ).
\end{equation} Each summand of this decomposition is a simple algebra and abelian only when $S_\lambda(V)$  is 1--dimensional.  This happens only if $n=im$ and $S_\lambda(V)=\bigwedge^m(V)^{\otimes i}$. In this case 
\begin{equation}\label{nodd}
(End_F(V)^{\otimes n})^{S_n}_{ab}=End(\bigwedge^m(V)^{\otimes i}),\ N(a)=\det(a)^i,\ a\in End_F(V).
\end{equation}
By the remarkable theory of {\em good filtrations}  or combinatorially of standard double tableaux, in positive characteristic or over the integers  the previous formula can be replaced by a canonical filtration of which Formula \eqref{swd} is the associated graded representation. From this one could deduce a general form of Formula \eqref{nodd}. 
\subsection{Determinants}
Given an algebra $R$ the composition of  a  homomorphism $j:R\to M_n(A)$ with the determinant $\det\circ j: R\stackrel{j}\longrightarrow  M_n(A)\stackrel{\det}\longrightarrow A$ is a multiplicative polynomial law, homogeneous of degree $n$.
One then obtains a natural transformation of functors:
\begin{equation}\label{nad}
 \hom_{\mathcal R_F}(R, M_n(A))\to \mathcal M_n(R,A),
\end{equation}
 and a commutative diagram.  \begin{equation}\label{unpp10}
\xymatrix{  R\ar@{->}[r]^{\mathtt j_R\qquad}\ar@{->}[rd]_f&M_n(T_n(R))\ar@{->}[d]^{M_n(\bar f)}\ar@{->}[r]^{\det}&T_n(R)\ar@{->}[d]^{ \bar f }\\
&M_n(A)\ar@{->}[r]^{\det}&A . }  
\end{equation}
 
    Since $T_n(R)$ is commutative, $\det\circ j_R$  factors  as 
$$ R\stackrel{i_R}\longrightarrow \Gamma_n(R)\stackrel{\pi}\longrightarrow \Gamma_n(R)_{ab} \stackrel{D}\longrightarrow T_n(R), \  D\circ\pi\circ i_R=\det\circ j_R.$$
The homomorphism $D:\Gamma_n(R)_{ab}  \longrightarrow T_n(R)$  is clearly invariant under the group of automorphisms $\mathfrak G_n$ so  it factors through $T_n(R)^{\mathfrak G_n}$.

\begin{remark}\label{pro}[Problem]
One of the main  problems of the theory is to understand when is that $D:\Gamma_n(R)_{ab} \to T_n(R)^{\mathfrak G_n}$  is an isomorphism.
\end{remark}  

This happens in several cases. In particular we have a fundamental result, Theorem 20.24 of \cite{depr0}.
\begin{theorem}\label{simm}
Consider the   free algebra $ A\langle X\rangle$  in some set of variables $X=\{x_i\}_{i\in I}$  with $A$  a field  or   the integers.

  Then   $D:\Gamma_n(A\langle X\rangle)_{ab} \to T_n(A\langle X\rangle)^{\mathfrak G_n}$  is an isomorphism.
\end{theorem} Notice that $T_n(A\langle X\rangle)^{\mathfrak G_n}$ is {\em the ring of invariants of $X$-tuples of $n\times n$ matrices}.

Theorem \ref{simm} is based on a Theorem of Procesi \cite{P3} and Razmyslov \cite{R} characterizing trace identities of matrices (cf. Theorem \ref{SFT}), and proved by Zieplies \cite{Ziplies} and Vaccarino \cite{Vacc},   when  $A=\mathbb Q$. 

The general case  is fully treated in \cite{depr0}. It is based on the characteristic free results of Donkin \cite{Don} and Zubkov \cite{Zub} on the invariants of matrices and a careful combinatorial study of  $\Gamma_n(A\langle X\rangle)$  inspired by the work of Zieplies.  

In fact since $A\langle X\rangle$ is a free $A$  module, its divided power is more conveniently described as the symmetric tensors:
$$\Gamma_n(A\langle X\rangle)\simeq (A\langle X\rangle^{\otimes n}) ^{S_n}.$$ Since, using the basis of monomials, the space $A\langle X\rangle^{\otimes n}$ is a permutation representation of $S_n$, one has a combinatorial description of  $(A\langle X\rangle^{\otimes n}) ^{S_n}.$
The abelian quotient, isomorphic to the ring of invariants of matrices, does not have a combinatorial description and it is a rather hard object to study.\smallskip

\begin{example}\label{ssf} If $X=\{x\}$ is a single variable we have that $A\langle X\rangle=A[x]$  is the commutative ring of polynomials, $A[x]^{\otimes n}=A[x_1,\ldots,x_n]$ so   $\Gamma_n(A[x])$  is the algebra of symmetric polynomials in $n$--variables, commutative. 

Consider the   elementary symmetric function $e_i$ defined by: \begin{equation}\label{esf}
(t-x)^{\otimes n}=\prod_{i=1}^n(t-x_i)=t^n-e_1t^{n-1}+e_2t^{n-2}-\ldots+(-1)^ne_n.
\end{equation}The map $D$  maps $(t-x)^{\otimes n}$ to $\det(t-\xi)= t^n+\sum_{i=1}^n(-1)^i\sigma_i(\xi)t^{n-i}$,  the  characteristic polynomial   of a generic matrix $\xi=(\xi_{i,j})$.

  So the elementary symmetric function $e_i$ maps  to $\sigma_i(\xi)$, the generators of invariants   of a single matrix.
  
  This is a very special case of Theorem \ref{simm}.
\end{example}

Although usually one deduces the functions $\sigma_i(x)$ for an $n\times n$ matrix $x$ from the determinant of $t-x$  one sees immediately that  embedding the $n\times n$ matrices into $n+1\times n+1$  matrices as upper left corner 
$$i_n(x):=\begin{vmatrix}
x&0\\0&0
\end{vmatrix},\ \det(t-i_n(x))=\det\begin{vmatrix}
t-x&0\\0&t\end{vmatrix}=\det(t-x)t $$ implies  $\sigma_i(x)=\sigma_i(i_n(x))$ so the functions $\sigma_i(x)$ are in fact also defined for infinite matrices with only finitely many non 0 entries $\bigcup_nM_n(F)$. It is thus important to understand the symbolic calculus on these functions and this is the theme of the next section.
 
\subsection{Symbolic approach}The construction   $R\mapsto \Gamma_n(R)$ is functorial in $R$ so given $r\in R$  the map $A[x]\to R,\ x\to    r$ induces a map $\Gamma_n(A[x])\to \Gamma_n(R)$ under which $$e_i\to \tau_i(r), \ (1+r)^{\otimes n}= 1+\tau_1(r)+\tau_2(r)+\ldots+\tau_n(r).$$ 

\begin{proposition}\label{labase}
It is proved,   in \cite{depr0} Lemma 20.12,   that, given a basis $a_i$ of $R$,  the elements $\tau_i(a_j)$  generate $(R^{\otimes n})^{S_n}$. Once we pass to the abelianization and set $\sigma_i(a)$ to be the class of $\tau_i(a )$  we have $\sigma_i(a b)= \sigma_i(ba ),\ \forall  a,b  $, see   \cite{depr0} Proposition  20.20.

\end{proposition}
We apply this construction to the free algebra. The grading of the algebra $A\langle X\rangle $  induces a grading of  $\Gamma_n(A\langle X\rangle)$.  
Recall that a monomial $M$ of positive length, is called {\em primitive} if it is not a power $N^k,\ k>1$. 
\smallskip

 From Proposition \ref{labase} and  Formula \eqref{fes} one sees that the elements $\tau_i(M),\ i=1,\ldots, n$   as    $M$  runs over the primitive monomials  generate  $\Gamma_n(A\langle X\rangle)$. The elements $\tau_i(M)$  satisfy complicated relations which are not fully understood. 
\begin{definition}\label{sna}
Denote by $\mathbb S_{n,A}$ the abelian quotient of $\Gamma_{n}(A\langle X\rangle)$ and by $\sigma_i(M)$  the class of $\tau_i(M)$ in $\mathbb S_{n,A}$.
\end{definition}  One can prove, Proposition 20.20 of \cite{depr0}, that if $M=AB$ one has $\sigma_i(AB)= \sigma_i(BA)$, one says that $AB$ and $BA$  are     {\em cyclically equivalent.} 

A {\em Lyndon word},    is a primitive monomial    minimal, in the lexicographic order,  in its class of cyclic equivalence.

As for the theory of symmetric functions one can pass to the limit, as $n\to\infty$, of the algebras $\Gamma_n(A\langle X\rangle)$ and their abelian quotients.

 If $\epsilon :      A\langle X\rangle\to A$ is the evaluation of  $X$ in 0, we have the map  
$$\pi_n:A\langle X\rangle^{\otimes n+1}\to A\langle X\rangle^{\otimes n},\ \pi_n(a_1 \otimes \ldots\otimes a_n\otimes a_{n+1})=a_1 \otimes \ldots\otimes a_n\otimes \epsilon (a_{n+1}). $$
This induces a map, still called $\pi_n: \Gamma_{n+1}(A\langle X\rangle)\to \Gamma_n(A\langle X\rangle) $. We have  $$\pi_n(\tau_i(M))=\begin{cases}
 \tau_i(M) \ \text{if}\ i\leq n\\
0\ \text{if}\ i= n+1
\end{cases}.$$ 
\begin{definition}\label{definf}
One can then define a limit algebra $ \Gamma_{\infty}(A\langle X\rangle)$ generated by the elements $\tau_i(M),\ i=1,\ldots, \infty$   as    $M$  runs over the primitive monomials and its abelian quotient $\mathbb S_{ A}\langle X\rangle$ denoted often just $\mathbb S_{ A}$, generated by the classes  $\sigma_i(M)$ of $\tau_i(M)$.   

The maps $\pi_n$ give rise to limit maps:
\begin{equation}\label{lima}
\tau_n: \Gamma_{\infty}(A\langle X\rangle)\to \Gamma_n(A\langle X\rangle),\  \tau_n: \mathbb S_{ A}\langle X\rangle\to \mathbb S_{n, A}\langle X\rangle.  \end{equation}

\end{definition}
In \cite{depr0} we have proved:

 {\bf Corollary 20.15} 
The algebra $\Gamma_{\infty}(\mathbb Q\langle X\rangle)$ is the universal enveloping algebra of 
$\mathbb Q_+\langle X\rangle$ considered as a Lie algebra.\smallskip

As for the structure of $\mathbb S_{n,A}$,
   Theorem 20.22 of \cite{depr0} (due to Zieplies) states that, in the commutative algebra  $\mathbb S_A$ one has $\sigma_i(AB)=\sigma_i(BA)$ for all monomials $A,B$  and finally  that  $\mathbb S_{ A}=A[\sigma_i(M)],\ M$ a Lyndon word, is the free polynomial ring in the variables $\sigma_i(M)$ as $M$ varies among the Lyndon words. 

Let $T_A(X)$  denote the monoid of endomorphisms  of $A\langle X\rangle$ given by mapping each variable $x_i\in X$ to some element $f_i\in A\langle X\rangle_+$, the ideal kernel of $\epsilon$ of elements with no constant term, this condition means that these endomorphisms commute with the map $\epsilon$. Each such endomorphism induces an endomorphism of each $\Gamma_n(A\langle X\rangle) $  compatible with the maps $\pi_n$ and hence an endomorphism on $\mathbb S_{n,A}:=\Gamma_n(A\langle X\rangle) _{ab}$ and on $\Gamma_{\infty}(A\langle X\rangle)$ and $\mathbb S_A.$ 
\begin{definition}\label{Tid}
A $T$--ideal of  $A\langle X\rangle$ or of $\Gamma_n(A\langle X\rangle),$ or $\mathbb S_{n,A}, \ n=1,\ldots,\infty $ is a multigraded ideal  $I$ closed under all endomorphisms induced by $T_A(X)$.
\end{definition}
\begin{remark}\label{ple} The condition of $I$  to be multigraded  can be replaced by the condition that, for every commutative $A$ algebra $B$, the ideal $_BI:=B\otimes_AI$  is  closed under all endomorphisms induced by $T_B(X)$.
This is in the spirit of polynomial laws.
\end{remark}

 For each $i=1,2,\ldots$   we have the maps  \begin{equation}\label{silaw}
f\mapsto \tau_i(f),\ A\langle X\rangle_+\to\Gamma_{\infty}(A\langle X\rangle);\quad  f\mapsto \sigma_i(f),\ A\langle X\rangle_+\to \mathbb S_A.
\end{equation}   They  are both  polynomial laws homogeneous of degree $i$ which commute with the action of the endomorphisms $T(X)$.  
 
 There is an explicit Formula, see \cite{depr0} Theorem 4.15 p. 37, which allows us to compute these laws for $\mathbb S_A$. It is due to Amitsur, \cite{Am},  (who stated it for matrix invariants), and later independently by Reutenauer and Sch\"utzenberger \cite{Reut1}. 
 \begin{theorem}\label{AmF}
Given an  $n\in\mathbb N$, non commutative variables $x_i$ and commutative parameters $t_i$:
\begin{equation}\label{comAm}
\sigma_n(\sum_it_ix_i)=\!\!\!\!\!\!\!\!\!\!\!\sum_{\begin{matrix}
(p_1<\ldots <p_k)\subset  W_0,\\  j_1,\ldots,j_k\in \mathbb N,\ \sum j_i\ell(p_i)=n  \end{matrix}}\!\!\!\!\!\!\!\!\!\!\!\!\!\!\!\!\!\!\!(-1)^{n-  \sum j_i}
t^{\sum_{i=1}^k j_i\nu(p_i)}\sigma_{j_1}(p_1)\ldots \sigma_{j_k}(p_k)
\end{equation} \end{theorem}
Here $W_0$  denotes the set of Lyndon words ordered by the degree lexicographic order. For a word $p$,  $\nu(p)$ is the vector $(a_1,\ldots,a_n)$  with $a_i$ counting how many times the variable $x_i$ appears in $p$. Finally $t^{(a_1,\ldots,a_n)} :=\prod_{i=1}^nt_i^{a_i}.$\medskip

In particular one can collect, in Formula  \eqref{comAm}  the terms of the same degree in the variables $x_i$ and have an explicit expression of the {\em polarized} forms of $\sigma_n(x)$:
\begin{equation}\label{comAm1}
\sigma_n(\sum_it_ix_i)= \sum_{ (a_1,\ldots,a_n)\mid \sum_ia_i=n}\prod_{i=1}^nt_i^{a_i}
 \sigma_{n;a_1,\ldots,a_n}(x_1,\ldots,x_n).
\end{equation}

 Substituting for a variable $x$ a linear combination   $\sum_jt_jM_j$  of monomials and applying Formula \eqref{comAm} to $\sigma_n(\sum_jt_jM_j)$ one obtains an element of $\mathbb S_A$ provided one has a further law. In fact a primitive word computed in monomials need no more be primitive  so we also need   the expression of  the elements $\sigma_i(x^j)$  in terms of the $\sigma_k(x), k\leq i\cdot j.$ These Formulas  arise from Example \ref{ssf}  as  the (stable) universal polynomial formulas in the algebra of symmetric functions expressing  \begin{equation}\label{fes}
e_i(x_1^j,\ldots, x_n^j)=P_{i,j}(e_1,\ldots,e_{i\cdot j})\end{equation}  in terms of the elementary symmetric functions  $e_k(x_1,\ldots, x_n )$  for $n>i\cdot j$.
 
 One has
 \begin{theorem}\label{quo}
The Kernels of the maps $\Gamma_{\infty}(A\langle X\rangle)\to \Gamma_{n}(A\langle X\rangle)$, respectively $ \mathbb S_A \to \mathbb S_{n,A} $ are the $T$--ideals generated by all the elements $\tau_i(f),\ i>n$, respectively  the $T$--ideal  generated by all the elements $\sigma_i(f),\ i>n,\ f\in \ A\langle X\rangle $. 
\end{theorem} 

In other words, in the case  $\mathbb S_{n,A}$,  the Kernel of $\pi_n$ is the ideal generated by all the polarized forms $ \sigma_{m;a_1,\ldots,a_n}(p_1,\ldots,p_m), m>n$  with $p_1,\ldots,p_m$  monomials.

The Theorem of Zubkov then states that the ring of invariants  of matrices has the same generators and relations as $\mathbb S_{n,A}$, when $A=\mathbb Z$ or a field, hence the isomorphism of  Theorem \ref{simm}.

The following example shows for $n=2$  an explicit deduction of the multiplicative nature of the determinant  from these relations.
\begin{align*}
& \sigma_{3;1,1,1}(a,b,ba)=\\& +\sigma_1(a) \sigma_1(b) \sigma_1(ab ) -\sigma_1(a)\sigma_1(ab^2 )-\sigma_1(b)\sigma_1(a^2b) +\sigma_1(a^2b^2)-2\sigma_2(ab)\quad\quad\\
 &\sigma_{4;2,2}(a,b)=\\&  - \sigma_1(a)\sigma_1(b) \sigma_1(ab)+  \sigma_1(a) \sigma_1(ab^2) + \sigma_1(b) \sigma_1(a ^2b) -\sigma_1(a^2b^2)\\ &+\sigma_2(ab)+\sigma_2 (a) \sigma_2(b) \\& \sigma_{3;1,1,1}(a,b,ba)+\sigma_{4;2,2}(a,b)=\sigma_2(ab)-\sigma_2(a) \sigma_2(b). 
\end{align*}
One can in fact take the basic identity for invariants of $n\times n$ matrices.
\begin{equation}\label{bbd}
\det(ab)=\det(a)\det(b)\iff \sigma_n(ab)-\sigma_n(a) \sigma_n(b)=0.
\end{equation}
Consider the  polynomial ring  $A[\sigma_i(p)],\ p\in W_0,\ i\leq n$.  Using Formula \eqref{comAm}   define, for each $f=\sum_it_iM_i\in A\langle X\rangle $, the element  $\sigma_k(f),\ k\leq n$ as follows.

If one substitutes each $x_i$ with $M_i$ in  Formula \eqref{comAm} one has a formal expression containing symbols $\sigma_i(M),\ i\leq k$  where $M$  may be an arbitrary monomial (including 1).  Then $M$ is cyclically equivalent to some power $N^j$ with $N\in W_0$ a Lyndon word.  One then considers  in the algebra of symmetric functions in exactly $n$ variables the Formula  \eqref{fes} with $e_i=0,\ \forall i>n$  i.e.  $e_i(x_1^j,\ldots, x_n^j)=P_{i,j}(e_1,\ldots,e_n,0,\ldots,0)$  in terms of the elementary symmetric functions  $e_k(x_1,\ldots, x_n ),\ k\leq n$. One then sets        \begin{equation}\label{pr}
\sigma_i(N^j)=P_{i,j}(\sigma_1(N ),\ldots,\sigma_n(N )) ,\ \sigma_i(1):=\binom n i .
\end{equation} 

Given $f=\sum_iu_iM_i,\ g=\sum_iv_iM_i\in A\langle X\rangle $ one may consider
 \begin{equation}\label{sec}
\sigma_n(fg)-\sigma_n(f) \sigma_n( g)=\sum_{\underline h,\underline k} u^{\underline h} v^{\underline k}\phi_{\underline h  ,\underline k},\  \phi_{\underline h  ,\underline k}\in A[\sigma_i(p)].
\end{equation} Evaluating  the variables $\xi_i\in X$  in the generic $n\times n$ matrices one has a homomorphism $\rho:A\langle X\rangle \to A[\xi_i]$ to the algebra of generic matrices which extends to a homomorphism of the symbolic algebra $\rho: A[\sigma_i(p)]\to A[\xi_{i,j}^k]^{PGL(n)}$  to the ring of invariants of matrices. By the Theorem of Donkin this is surjective. Moreover clearly the identity given by \eqref{pr}  holds for the corresponding matrix invariants. As for \eqref{sec} we have that $\sigma_n(\rho(fg))-\sigma_n(\rho(f)) \sigma_n( \rho(g))=\det(\rho(f)\rho(g))-\det(\rho(f)) \det( \rho(g))=0$  so all the elements $\phi_{\underline h  ,\underline k}$ map to 0.
\begin{theorem}\label{nZ}  The Kernel of $\rho$  is the ideal $K$ of  $A[\sigma_i(p)]$ generated by the elements $ \phi_{\underline h  ,\underline k}$  of Formula \eqref{sec}, when computed using Formula \eqref{pr} and \eqref{comAm}. 
\end{theorem}   
\begin{proof} The previous relations express   the identity $\sigma_n(fg)=\sigma_n(f) \sigma_n( g)$.  Consider the  algebra $A[\sigma_i(p)]/K$, and the map $A\langle X\rangle \to A[\sigma_i(p)]/K$ mapping $f\in A\langle X\rangle$ to the class $\bar \sigma_n(f)$  of $ \sigma_n(f)$ modulo $K$.

By construction this is a multiplicative map homogeneous of degree $n$ so it factors through a map  $A\langle X\rangle \to \Gamma_n(A\langle X\rangle)_{ab}\stackrel{\bar\rho}\to A[\sigma_i(p)]/K$. On the other hand  $\Gamma_n(A\langle X\rangle)_{ab}$ is generated by the elements  $\sigma_i(p)$  and the generators of $K$    are 0 in $\Gamma_n(A\langle X\rangle)_{ab}$ hence $\bar\rho$ is an isomorphism and so the claim follows from Theorem \ref{simm}.

\end{proof}

\subsection{The free $n$--Cayley--Hamilton algebra\label{ffch}}
$n$--Cayley--Hamilton algebras over a commutative ring $F$ form a category, where a map  is assumed to commute with the norm.

If the base ring  $F$ is either $\mathbb Z$ or a field, one has a particularly useful description of the  free $n$--Cayley--Hamilton algebra in any set of variables $X=\{x_i\}_{i\in I}$. It will be given in Corollary \ref{frch}.\smallskip

Recall the definition \ref{genmad} of  $F\langle \xi_i\rangle\subset M_n( F[\xi_{h,k}^{(i)}]),\ i\in I,\ h,k=1,\ldots, n  $ the algebra of generic matrices, and the action  of $PGL$  on $ M_n( F[\xi_{h,k}^{(i)}])$ and on  $ F[\xi_{h,k}^{(i)}]$. The algebra $ F[\xi_{h,k}^{(i)}]^{PGL}$ is by definition the algebra of   invariants   of $X$--tuples of $n\times n$ matrices. 
The algebra $ M_n( F[\xi_{h,k}^{(i)}])^{PGL}$ is by definition the algebra of   equivariant maps (polynomial laws)  from    $X$--tuples of $n\times n$ matrices to  $n\times n$ matrices. 

One starts from the free algebra $F\langle x_i\rangle_{ i\in I}$ and the map  to  the generic matrices  $j:=j_{F\langle x_i\rangle}:F\langle x_i\rangle\to F\langle \xi_i\rangle\subset M_n( F[\xi_{h,k}^{(i)}])$, and then compose this with the determinant
$$\begin{CD}
 {F\langle x_i\rangle} @>j_{F\langle x_i\rangle}>> F\langle \xi_i\rangle @>i>>M_n( F[\xi_{h,k}^{(i)}])@>\det >>F[\xi_{h,k}^{(i)}]^{PGL}
\end{CD} $$
\begin{theorem}\label{DoZu}\begin{enumerate}\item In the commutative diagram
\begin{equation}\label{unpp3}
\xymatrix{  F\langle x_i\rangle_{ i\in I}\ar@{->}[r]^{r\mapsto r^{\otimes n}}\ar@{->}[d]_{j_{F\langle x_i\rangle}}&\Gamma_n(F\langle x_i\rangle_{ i\in I}) \ar@{->}[d]^{  N}\ar@{->}[r]^\pi&\Gamma_n(F\langle x_i\rangle_{ i\in I})_{ab}\ar@{->}[d]^{\simeq  \bar N}\\
F\langle \xi_i\rangle\ar@{->}[r]^{\det} &F[\xi_{h,k}^{(i)}]^{PGL}
\ar@{->}[r]^1&F[\xi_{h,k}^{(i)}]^{PGL}
}  
\end{equation} the last map  $\Gamma_n(F\langle x_i\rangle_{ i\in I})_{ab}\stackrel{\bar N}\to F[\xi_{h,k}^{(i)}]^{PGL}$ is an isomorphism.
\item Extend the map $\det\circ j_{F\langle x_i\rangle}$  to a norm $F\langle x_i\rangle_{ i\in I} \otimes_F F[\xi_{h,k}^{(i)}]^{PGL}\to F[\xi_{h,k}^{(i)}]^{PGL}$.  Then this induces a norm compatible homomorphism $$\bar{j}_{F\langle x_i\rangle}\!\!:F\langle x_i\rangle  \otimes_F F[\xi_{h,k}^{(i)}]^{PGL}\stackrel{j  \otimes 1}\to  F\langle \xi_i\rangle  \otimes_F F[\xi_{h,k}^{(i)}]^{PGL}\stackrel{m}\to  M_n( F[\xi_{h,k}^{(i)}])^{PGL} , $$
 with $m$ the multiplication.\item  $\bar{j}_{F\langle x_i\rangle}$ is surjective and its kernel $K$ is generated by the evaluation of the $n$ characteristic polynomial  of all its elements.
  \end{enumerate}

\end{theorem}

This is proved in \cite{depr0},Theorem 18.17  and Remark 18.18 based on  the theorems of Donkin \cite{Don} and Zubkov, \cite{Zub}. 
\begin{remark}\label{sftg}
\begin{enumerate}\item If the set of  variables has $\ell $ elements, then  the algebra $F[\xi_{h,k}^{(i)}]^{PGL}$ equals  the algebra $T_n(\ell)$  and the algebra $M_n( F[\xi_{h,k}^{(i)}])^{PGL}$ equals $\mathcal S_n(\ell)$  of  page \pageref{prSl}.
\item 
Using Theorem \ref{quo}
 part 3. can be equivalently stated as follows:
 
Consider    the map $\rho_n:  \mathbb S_F\langle x_i\rangle_{ i\in I}\stackrel{ \tau_n}\to   \mathbb S_{n,F }\langle x_i\rangle_{ i\in I}\stackrel{ \tilde N}\to    F[\xi_{h,k}^{(i)}]^{PGL}$  and:\begin{equation}\label{eqac}
 F\langle x_i\rangle_{ i\in I} \otimes_F\mathbb S_F\langle x_i\rangle_{ i\in I}\stackrel{\tau_n\otimes \rho_n}\to  F\langle \xi_i\rangle_{ i\in I} \otimes_F F[\xi_{h,k}^{(i)}]^{PGL}\stackrel{m}\to  M_n( F[\xi_{h,k}^{(i)}])^{PGL} .
\end{equation}This map is surjective and its kernel is the ideal generated by   the evaluation of the $m$ characteristic polynomials  of all its elements for all $m>n$ and all the corresponding evaluations of the $\sigma_i$ in this ideal.\end{enumerate}

\end{remark}

As a Corollary one has
\begin{corollary}\label{frch} The algebra $M_n( F[\xi_{h,k}^{(i)}])^{PGL}$ is a free algebra on the generators $\xi_i$  in the category of  $n$ Cayley--Hamilton $F$--algebras.

\end{corollary}
\begin{proof}
Let $R$ be an  $n$ Cayley--Hamilton $F$--algebra, with norm algebra $A$ and consider a set $r_i,\ i\in I$ of elements of $R$.  

Then one deduces a homomorphism $f:F\langle x_i\rangle_{ i\in I} \to R,\ x_i\mapsto r_i$. Next one has a commutative diagram
$$\begin{CD}
F\langle x_i\rangle @>f>> R \\@VVV@VVV\\\Gamma_n(F\langle x_i\rangle) @>\bar f>>\Gamma_n(R)@>N >>A
\end{CD} $$
From this one has a homomorphism $F[\xi_{h,k}^{(i)}]^{PGL}\to A$ and one of norm algebras $\bar f:F\langle x_i\rangle_{ i\in I} \otimes_F F[\xi_{h,k}^{(i)}]^{PGL}\to R$. 

Clearly  (by Theorem \ref{DoZu}   3)) $\bar f $  vanishes on the kernel $K$  of the quotient map $\bar\pi_n:F\langle x_i\rangle_{ i\in I} \otimes_F F[\xi_{h,k}^{(i)}]^{PGL}\to  M_n( F[\xi_{h,k}^{(i)}])^{PGL}$ so $\bar f$ factors through a norm compatible map $ M_n( F[\xi_{h,k}^{(i)}])^{PGL} \to R$.
\end{proof}
\begin{corollary}\label{frch1} Every  $n$ Cayley--Hamilton $F$--algebra $R$ is the quotient,   in the category of  $n$ Cayley--Hamilton $F$--algebras, of an algebra $M_n( F[\xi_{h,k}^{(i)}])^{PGL}$ in some  generators $\xi_i,\ i\in I$.

In particular every  $n$ Cayley--Hamilton $F$--algebra $R$ satisfies all polynomial identities with coefficients in $F$ of $n\times n$ matrices.

\end{corollary}

\subsection{Azumaya algebras}
An important class of CH--algebras are Azumaya algebras, \cite{Az}, \cite{A-G2}.  For our purpose we may take as definition:
\begin{definition}\label{azza}
An   algebra  $R$ with   center $Z$ is Azumaya of rank $n^2$ over   $Z$ if there is a faithfully flat  extension $Z\to W$ so that $W\otimes_ZR\simeq M_n(W)$.
\end{definition}    
Then it is easily seen that the determinant,  of the matrix algebra  $M_n(W)$  restricted to $R$  maps to $Z$  giving rise to a multiplicative map called {\em the reduced norm $N:R\to Z.$}  Then $R$, with this norm,  is an $n$--CH algebra.\smallskip

We want to see that this is essentially the only norm on $R$.

The next Theorem is due in part to Ziplies (but his proof is quite complicated and very long) \cite{Zip}.
\begin{theorem}\label{nazz}
Let $R$ be a  rank $n^2$  Azumaya algebra over $Z$. $ (R^{\otimes m})^{S_m}_{ab}=0$ if $m$ is not a multiple of $n$.  If $m=in$ then the map $ \bar N^i:(R^{\otimes m})^{S_m}_{ab}\to Z$, 
induced by the reduced norm $N^i:R\to Z,$  is an isomorphism.\end{theorem}
  \begin{proof}
By faithfully flat descent we may reduce to the case  $R=M_n(Z)$.

We first show the statement for $(M_n(\mathbb Z)^{\otimes n})^{S_n}_{ab}$.  Apply Proposition \ref{labase} to the basis of elementary matrices $e_{i,j}$. We want to prove that for all   $e_{i,j}$ we have  $\sigma_h(e_{i,j})\in \mathbb Z$. If $i\neq j$  we have $$\sigma_h(e_{i,j})=\sigma_h(e_{i,i}e_{i,j})=\sigma_h(e_{i,j}e_{i,i})=0,$$ $$ \sigma_h(e_{i,i})=\sigma_h(e_{i,j}e_{j,i})=\sigma_h(e_{j,i}e_{i,j})=\sigma_h(e_{j,j})  .$$ 
Apply now Amitsur's Formula \eqref{comAm}   
  to  the  $\sigma_j(C)$, $C$ the permutation matrix    $C:=e_{1,2}+e_{2,3}+\ldots+e_{n-1,n}+e_{n,1}$, of the full cycle $ (1,2,\ldots,n)\in S_n$. 
  
The only non zero monomials in the $x_i=e_{i,i+1}$   have either value  some $e_{i,j},\ j\neq i$ or some $e_{h,h}$, but of these the only Lyndon word is $x_1x_2\ldots x_n=e_{1,1}$  of degree $n$.  

Since $\sigma_j(a)=0,\ \forall j>n$ we deduce
$$ \sigma_i( e_{1,1})= \sigma_i( x_1x_2\ldots x_n)=(-1)^{ (n- i)}\sigma_{i\cdot  n}(e_{1,2}+ \ldots+e_{n-1,n}+e_{n,1})=0,  \forall i>1.$$
We claim that $ \sigma_1( e_{1,1})=(-1)^{n-1}\sigma_n(C)=1$, which completes the computation.  

Now $\sigma_n$ is a multiplicative map and so it is 1 on the     alternating group $A_n\subset S_n$.  If $n$ is odd then $C\in A_n$    so $\sigma_1( e_{h,h})=1$. If $n=2k$ then $C^2\in A_n$.   Set $a:=\sigma_n(C )$, so $ \sigma_1( e_{1,1})=-a$ , we have  $a^2=\sigma_n(C^2)=1$.

Let $A=-e_{1,1}+\sum_{i=2}^ne_{i,i}$

 Apply to $A$ Formula \eqref{comAm}, a primitive monomial in the terms of $A$ vanishes unless it is   one  $e_{j,j}.$ So 
$$\sigma_{ n}(A)=-\prod_{h=1}^{2k}\sigma_1( e_{h,h}) =-a^{2k}=-1.$$  
Now $\det(AC)=1$  so if $n\geq 4$ it  is a product of commutators so 
$$1=\sigma_{ n}(AC)=-a\implies \sigma_1( e_{h,h})=1.$$
For $n=2$  one may check directly that $$ A=(1+e_{1,2})(1+e_{2,1}) (1-e_{1,2})(1-e_{2,1}) =3 e_{1,1}-e_{1,2}+e_{2,1}$$ is a commutator  and deduce from Amitsur's Formula for $1=\sigma_2(A)$:
$$1=-\sigma_{1}(-e_{1,2} e_{2,1}) =\sigma_{1}( e_{1,1}) .$$

Next  consider $(M_n(\mathbb Z)^{\otimes k})^{S_k}_{ab}$ for $k$ not a multiple of $n$.

The same argument shows that in this case $\sigma_k(C)=0$.  But $C^n=1$ and $\sigma_k$ is multiplicative  so $1=\sigma_k(C^n)=0$.

The case $k=in$  seems to be difficult to attack with this method so we use a different approach.
\end{proof}   
\begin{lemma}\label{multpo}
Let $F$ be an algebraically closed field, $A$   a  commutative algebra over $F$ and $f:M_n(F)\to A$  be a multiplicative polynomial map of degree $m$. Then $m=in$ and $f(a)=\det(a)^i1_A.$
\end{lemma}
\begin{proof}
We  have already shown that multiplicative maps can exist only for degrees multiples of $n$ so assume that $m=in$.  For $\lambda$ a scalar matrix, since $f(1)=1_A$  one must have $f(\lambda)=\lambda^m1_A$. If $a\in SL(n,F)$  then $a$ is a product of commutators so that $f(a)=1_A$.

If $a\in GL(n,F)$ we can write $a=b\lambda,\ b\in SL(n,F)$ and $\lambda$ a scalar so that 
$f(a)=\lambda^m1_A=\det(a)^i1_A.$ Since $GL(n,F)$ is Zariski dense in $M_n(F)$ and $f$ is a polynomial it follows that for all matrices $a$ we have $f(a)=\det(a)^i1_A.$ 
\end{proof}
We claim  that 
$A:=(M_n(F)^{\otimes in})^{S_{in}}_{ab}=F.$ Let $\pi$ be the projection  of $(M_n(F)^{\otimes in})^{S_{in}}$ to $A$. If $a\in M_n(F)$ the map $a\mapsto \pi(a^{\otimes in})$  is a multiplicative polynomial map  so it is $\det(a)^i1_A$ and maps  $M_n(F)$ to $F1_A$. Now the elements $a^{\otimes in}$ span $(M_n(F)^{\otimes in})^{S_{in}}$  and by construction $\pi$ is surjective so the claim follows.

Now let us pass to the general case $A:=(M_n(\mathbb Z)^{\otimes in})^{S_{in}}_{ab}.$  We have $A  =(M_n(\mathbb Z)^{\otimes in})^{S_{in}}/J$  where $J$ is the ideal generated by commutators. The algebra $A$ as abelian group if finitely generated. For any field $F$  we have 
 the exact sequence
 \begin{equation}\label{perd}
\begin{CD}
\quad 0@>>>J@>>>(M_n(\mathbb Z)^{\otimes in})^{S_{in}}@>>>A@>>>0\quad\\@.
F\otimes J@>i>>F\otimes (M_n(\mathbb Z)^{\otimes in})^{S_{in}}@>\pi>>F\otimes A@>>>0
\end{CD}
\end{equation}  Now $F\otimes (M_n(\mathbb Z)^{\otimes in})^{S_{in}}= (M_n(F)^{\otimes in})^{S_{in}}$ and $i(F\otimes J)$ is the ideal of $(M_n(F)^{\otimes in})^{S_{in}}$  generated by commutators and $\pi$ is surjective  so by the previous Lemma $F\otimes A=F$.

Since this is true for all $F$ of all characteristics  one must have $A=\mathbb Z$. 

In particular the first exact sequence splits and so for all commutative rings  $B$ one has
\begin{equation}\label{perd1}
\begin{CD}
0@>>>B\otimes J@>i>>(M_n(B)^{\otimes in})^{S_{in}}@>\pi>>B@>>>0
\end{CD}
\end{equation} and
$$ B= (M_n(B)^{\otimes in})^{S_{in}}_{ab}.$$\bigskip
 
\paragraph{Azumaya algebras and invariants}For  Azumaya algebras the Problem  of Remark \ref{pro} has a positive answer:\begin{theorem}\label{inaz}
If $R$ is a rank $n^2$ Azumaya algebra, over its center $Z$ the map  $D:Z=\Gamma_n(R)_{ab} \to T_n(R)^{\mathfrak G_n}$  is an isomorphism.
\end{theorem}

Assume first that $R=M_n(A)$. A map  $M_n(A)\to M_n(B)$ consists of a morphism $f:A\to B$ and then an automorphism   $g:B\otimes_ZM_n(A)\to  M_n(B)$, this functor is thus classified by $T_n(M_n(A))=P_n\otimes A$  and we have the universal map $j:M_n(A)\to M_n(P_n\otimes A)\simeq P_n\otimes  M_n(A)$. The map $j$ is also given by Formula \eqref{cogm}:

 \begin{equation}\label{cogm1}
a\mapsto XaX^{-1}\in M_n(P_n\otimes A),\ a\in M_n(A).
\end{equation}
The condition for an element    $u\in T_n(M_n(A))=P_n\otimes A$ to be invariant is from Formula \eqref{finv}:
$$\Delta (u)=1\otimes u,\ \Delta:P_n\otimes A\to P_n\otimes P_n\otimes A ,\quad\text{comultiplication}$$
By the Hopf algebra properties if $S$ is the antipode we have $m\circ S\otimes 1\circ\Delta =\epsilon$ which in  group terms just means $f(x^{-1}x)=f(1)$
$$\implies m\circ S\otimes 1\circ  \Delta (u)=m\circ S\otimes 1(1\otimes u)=m(1\otimes u)=u=\epsilon(u)\in A.$$
 In order to analyze the universal map $R\to M_n(T_n(R)) $  for an Azumaya algebra $R$ we use a fact from the Theory of polynomial identities.
 
 One knows \cite{agpr}  Theorem 10.2.19, that there is a multilinear non commutative polynomial  $\phi(x_1,\ldots,x_k)$ with coefficients in $\mathbb Z$ which, when evaluated in matrices $M_n(A)$  over any commutative ring $A$  does not vanish and takes values in the center $A$. From this it follows that:
 \begin{lemma}\label{cenaz}
When we evaluate $\phi$ in any  Azumaya algebras $R$  of rank $n^2$ over its center $Z$  its values lie in $Z$  and every element of $Z$  is obtained as a sum of evaluations of $\phi$. 
\end{lemma}\begin{proof}
In fact from the faithfully flat splitting $W\otimes_ZR\simeq M_n(W)$ it follows that $\phi$ takes values in $Z$ and it does not vanish on $R$  otherwise it would vanish on $M_n(W)$. Since $\phi$ is multilinear the sums of evaluations of $\phi$ form an ideal $J$ of $Z$  and, if $J\neq Z$, then $\phi$  vanishes on  the Azumaya algebras $R/JR$  of rank $n^2$ over its center $Z/J $ a contradiction.
\end{proof} 
 \begin{lemma}\label{maaz}
If $f:R_1\to R_2$ is a ring homomorphism of two   rank $n^2$ Azumaya algebras  with centers $Z_1, Z_2$  then $f(Z_1)\subset Z_2,\ R_2\simeq Z_2\otimes_{Z_1}R_1$.
\end{lemma}
\begin{proof}
Given $a\in Z_1$  we have $a$ is a sum of evaluations $\phi(a_1,\ldots,a_k)$ for some elements $a_i\in R_1$ so $f(a)$ is a sum of evaluations $\phi(f(a_1),\ldots,f(a_k))\in Z_2.$

\end{proof}
 We have  that $T_n(R)$ classifies the functor  $\hom_{\mathcal R}(R, M_n(B))$ and    under such a map $f:R\to M_n(B)$ we have $f(Z)\subset B$ and $B\otimes_ZR\simeq M_n(B)$ which is an isomorphism.   We may, to begin with,   assume that instead of working in the category of   rings we work in that of   $Z$ algebras $\mathcal Z$.  Then $\hom_{\mathcal Z}(R, M_n(B))$ is the set of isomorphisms $g:B\otimes_ZR\simeq M_n(B)$.
 
 This set, for a given $B$,  may be empty  but when it is non empty we say that $B$ {\em splits} $R$  and  on it the group  $PGL(n,-)$ acts in a simply transitive way  that is, by definition,  we have  that $T_n(R)$ is a {\em torsor}  for $PGL(n,-)$.
 \begin{definition}\label{tors}  Let $\mathcal C$  be a category and $G(-), F(-)$  be respectively a covariant group valued and a set valued functor, together with a group action $\mu:G(-)\times F(-)\to F(-)$.
 
 We say that $F$ is a {\em torsor} over $G$ if for all $X\in \mathcal C$  given $x,y\in F(X)$  there is a unique element $g\in G(X)$ with $gy=x$.
 
 If $G$, $F$  are represented by two objects $\mathtt G,\ \mathtt F$ and $\mathcal C$ has products we say that  $\mathtt F$ is a torsor over the group like object $ \mathtt G $  (cf. \cite{Mil}).

\end{definition} Notice again that when $F(X)=\emptyset$ the condition is void.

The previous condition can be conveniently reformulated as: 

{\em the natural transformation of functors:\begin{equation}\label{ntt}
\begin{CD}
G(X)\times F(X)@>1\times\delta>>G(X)\times F(X)\times F(X)@>\mu\times 1>> F(X)\times F(X)
\end{CD}
\end{equation} is an isomorphism (with $\delta$ the diagonal)}.\smallskip

 One knows that there is a faithfully flat extension $j:Z\to W$  so that $W\otimes_ZR\simeq M_n(W)$.  One deduces  for the universal object $T_n(R)$  that $$W\otimes_Z T_n(R) \simeq W\otimes P_n $$ and so that $T_n(R) $ is   faithfully flat  over $Z$. 
 
 From Formula \eqref{ntt}  the property of being a torsor can be stated in terms of the coaction of $P_n\otimes Z$  over $T_n(R)$, as in Formula \eqref{coac}: $$\eta:T_n(R)\to (P_n\otimes Z)\otimes_ZT_n(R)= P_n\otimes T_n(R).$$   Since $F$ and $G$ are controvariant functors     Formula \eqref{ntt}  gives a    dual formula for their representing objects. This  is the fact that the map $1\otimes m\circ \eta\otimes 1$ (where $m$ is the multiplication $m(a\otimes b)=ab$) is an isomorphism:
  $$1\otimes m\circ \eta\otimes 1 :T_n(R) \otimes T_n(R)\stackrel{\eta\otimes 1}\longrightarrow P_n \otimes T_n(R) \otimes T_n(R)\stackrel{1\otimes m}\longrightarrow P_n \otimes T_n(R).$$ This formula at the level of points, in any commutative $Z$--algebra $A$,       means exactly that given two points $(x,y)$ in $\hom(T_n(R),A)$  there is a unique $g\in \mathfrak G(A)$ with $x=gy,$ or $ (x,y)=(gy,y)$.

If $u\in T_n(R)$  we have 
$$ 1\otimes m\circ \eta\otimes 1(u\otimes 1)=1\otimes m\circ \eta (u)\otimes 1= \eta (u)$$
$$ 1\otimes m\circ \eta\otimes 1(1\otimes u)=1\otimes m\circ 1\otimes u=1\otimes u.$$
Since the map  $ 1\otimes m\circ \eta\otimes 1$ is an isomorphism,  the condition of invariance  $\eta (u)=1\otimes u$ is equivalent to $u\otimes 1 =1\otimes u$ which by faithfully flat descent is equivalent to $u\in   Z$. This proves Theorem \ref{inaz}.

\subsection{$\Sigma$--algebras}  We have  seen in Formula \eqref{silaw} the maps
$\sigma_i(f),\ A\langle X\rangle_+\to \mathbb S_A.
$ which  are    polynomial laws homogeneous of degree $i$ which commute with the action of the endomorphisms $T(X)$.  
We can view the operators  $\sigma_i$  as  homogeneous polynomial maps,  with respect to $  \mathbb S_A $, of degree $i$  of  $A\langle X\rangle_+\otimes \mathbb S_A $ to  the center $  \mathbb S_A $ which satisfy the Amitsur  identity \eqref{comAm}. Thus we may set the following definition:
\begin{definition}\label{sigal}
1)\quad A $\Sigma$ algebra $R$  is an algebra over a commutative ring $A$ equipped with  polynomial laws $\sigma_i:R\to R$ which satisfy:
 \begin{equation}\label{sigre}
[\sigma_i(a),b]=0,\  \sigma_i(\sigma_j(a)b)=\sigma_j(a)^i\sigma_i(b),\ \forall a,b\in _BR,\ \forall B\in\mathcal C_A.
\end{equation}
 \begin{equation}\label{fes1}
 \sigma_i(ab )= \sigma_i(ba),\quad  \sigma_i(a^j)\stackrel{\eqref{fes}}=P_{i,j}( \sigma_1(a),\ldots, \sigma_{i\cdot j}(a))\end{equation}  and that also satisfy   Amitsur's Formula \eqref{comAm}.

2)\quad The $\sigma$--algebra $\sigma(R)$ of $R$  is the algebra generated over  $A$  by the elements $\sigma_j(a),\ j\in\mathbb N, \ a\in R$.

It is a subalgebra  of the center of $R$  closed under the operations $\sigma_i$.
\end{definition}

\begin{definition}\label{siide}
An ideal $I\subset R$ in a $\Sigma$--algebra is a $\sigma$--ideal if it is closed under the maps $\sigma_i$.
\end{definition}
\begin{proposition}\label{sideq}
If $I$  is a $\sigma$--ideal of $R$  the maps $\sigma_i$ pass to the quotient $R/I$ which  is thus also a $\Sigma$--algebra.

 If $I$  is a $\sigma$--ideal of $R$ and $B$ is a commutative $A$ algebra, then    ${}_BI$  is a $\sigma$--ideal of ${}_BR.$ \end{proposition}
\begin{proof}
Given $a\in R$ and $b\in I$ we need to see that $\sigma_i(a+b)-\sigma_i(a)\in I$. This follows  from Amitsur's Formula.

The second part also follows from Amitsur's Formula.\end{proof}
\begin{remark}\label{spl} Since the $\sigma_i$ are polynomial laws  the previous statements extend to all $_BR$.

\end{remark}
Then  $\Sigma$--algebras also form a category, where $\hom_\Sigma (R,S)$  denotes the set of homomorphisms $f:R\to S$ commuting with the operations $\sigma_i$, i.e. $f(\sigma_j(r) ) =\sigma_j(f(r)),\ \forall j\in\mathbb N,\ \forall r\in R$. 

The kernel of a $\sigma$--homomorphism is a  $\sigma$--ideal  and the usual homomorphism theorem holds.

This category   has free algebras namely $A\langle X\rangle_+\otimes \mathbb S_A $ if we do not consider algebras with 1 or $A\langle X\rangle \otimes \mathbb S_A[\sigma_i(1)] $ by declaring the elements $\sigma_i(1)$ to be independent variables. 
\begin{definition}\label{sigid} A  $\Sigma$--identity for a  $\Sigma$--algebra $R$ is an element of the free algebra $f\in A\langle X\rangle \otimes \mathbb S_A$ which vanishes under all evaluations of $X$ in $R$.

\end{definition}  By abuse of notations we  denote by  $ \mathbb S_A $  the algebra  $ \mathbb S_A [\sigma_i(1)] ,\ i=1,\ldots.$\quad 
As in the Theory of polynomial identities one has then the notions of $T$--ideal of $A\langle X\rangle \otimes \mathbb S_A$, of variety of  $\Sigma$--algebras and of $\Sigma$ or PI equivalence of $\Sigma$--algebras.
 \begin{remark}\label{uapl}
1)\quad In this language  one defines a  formal Cayley Hamilton polynomial in the free algebra by $CH_n(x):=x^n+\sum_{i=1}^n(-1)^i\sigma_i(x)x^{n-i}$. 

Then an $n$--  Cayley Hamilton $A$--algebra $R$ can be also defined as a $\Sigma$ algebra satisfying the following conditions, which we will refer to as
\begin{definition}\label{tchc}
{\em three $CH_n$ conditions}:

1) $\sigma_i(x)=0,\ \forall i>n$, 2) $CH_n(x)=0,\ \forall i>n,\ \forall x\in {}_B R$, and $B$ any commutative $A$--algebra, and 3) $\sigma_i(1)=\binom ni, \forall i\leq n$. 
\end{definition} 
 \end{remark} 
\begin{remark}\label{sden}
It then follows from the Theorem of Zubkov and Zieplies that  $\sigma_n$ is a norm and  that $CH_n(x)$ is the evaluation for $t=x$ of $\sigma_n(t-x)$, see \cite{depr} and also \cite{Pr8}.

 \quad  In particular all the maps  $\sigma_i(x)$ are {\em deduced} from $\sigma_n( x)$ via the formulas $\sigma_n(t-x)=t^n+\sum_{i=1}^n(-1)^i\sigma_i(x)t^{n-i}  $ and $\sigma_i(x)=0,\ \forall i>n.$
\end{remark}

\subsubsection{The  first and second fundamental Theorem for matrix invariants {\em revisited}}
The first and second fundamental Theorem for matrix invariants for algebras  may be viewed as the starting point of the Theory of Cayley---Hamilton algebras, in all characteristics.  It is Theorem \ref{DoZu} which can be interpreted best in the language of $\Sigma$--algebras.
Let   $F=\mathbb Z$, or a field:
\begin{theorem}\label{SFT}
 The algebra $F_{\Sigma,n}\langle X\rangle$ of equivariant polynomial maps from $X$--tuples of $n\times n$ matrices, $M_n(F)^X$  to  $n\times n$ matrices $M_n(F)$,  is the free  $\Sigma$--algebra  $F\langle X\rangle \otimes \mathbb S_A$  modulo the $T$--ideal generated by the  three $CH_n$ conditions of Definition \ref{tchc}
  \begin{equation}\label{ffs}
{\boxed{ F_{\Sigma,n}\langle X\rangle:=F\langle X\rangle \otimes \mathbb S_A/\langle CH_n(x),\ \sigma_i(x)=0,\ \forall i>n,\ \sigma_i(1)=\binom ni  .} } 
\end{equation}

\end{theorem}
  To be concrete  if $X$ has $\ell$  elements, let $A_{\ell,n}$  denote  the polynomial functions on the space $M_n(F)^\ell $ (that is the algebra of polynomials over $F$ in $mn^2$ variables $\xi_{i,(j,h)},\ i=1,\ldots m;\ j,h=1,\ldots,n $). 

On this space, and hence on  $A_{\ell,n}$, acts the group $PGL(n,-)$  by conjugation.

The space of polynomial maps from $M_n(F)^\ell $ to $M_n(F)$ is
$$M_n(A_{\ell,n})=M_n(F)\otimes A_{\ell,n}.$$  This is a $\Sigma$--algebra in an obvious way, and  on this space acts diagonally $PGL(n,-)$  commuting with the $\Sigma$--operators and the invariants  
 $$F_{\Sigma,n}\langle x_1,\ldots,x_\ell \rangle =M_n(A_{\ell,n})^{PGL(n,-)}=(M_n(F)\otimes A_{\ell,n})^{PGL(n,-)}$$
are a  $\Sigma$--subalgebra which is  the relatively free algebra   in $\ell$ variables in the variety of $\Sigma$--algebras satisfying the three $CH_n$ conditions.

For the  $\sigma$--algebra of $F_{\Sigma,n}\langle X\rangle$ we have $T_n(\ell)=A_{\ell,n} ^{PGL(n,-)}$.  Of course we may let $\ell$ be also infinity (of any type)  and have $$F_{\Sigma,n}\langle X\rangle =M_n(A_{X,n})^{PGL(n,-)}=(M_n(F)\otimes A_{X,n})^{PGL(n,-)}$$ where $A_{X,n}$ is the polynomial ring on $M_n(F)^X.$
\begin{remark}\label{eqcc}  If $F$ is  an infinite field  one may take for  $PGL(n,-)$ the actual group $PGL(n,F)$, otherwise one has two options. The first option is to take  for  $PGL(n,-)$ the   group $PGL(n,G)$ for $G$ any infinite field containing $F$  or take the categorical notion, valid also over $\mathbb Z$, taking the coaction $\eta:A_{X,n}\otimes M_n(F)\to  P_n\otimes  A_{X,n}\otimes M_n(F)$ (with $P_n$ the coordinate ring of  $PGL(n,-)$)  set
$$(A_{X,n}\otimes M_n(F))^{PGL(n,-)}:=\{r\in (A_{X,n}\otimes M_n(F)) \mid \eta(r)=1\otimes r\}. $$
By the universal properties one sees that $\eta(r)=1\otimes r$ is equivalent to $g(r)=1\otimes r,\ \forall g\in PGL(n,B),\ \forall B.$\smallskip

The reader will understand at this point that the approach with trace and that with norm, in characteristic 0, are equivalent.
\end{remark}
For a proof of these Theorems   in all characteristics or even  $\mathbb Z$--algebras,  the Theorem of Zubkov,  the reader may consult \cite{depr0}.
\bigskip

One can then reformulate the definition of $n$--Cayley--Hamilton  algebra in this language:
\begin{definition}\label{CHalg}
A $\Sigma$--algebra      satisfying the three $CH_n$ conditions  \begin{equation}\label{tcn}
\langle CH_n(x),\ \sigma_i(x)=0,\ \forall i>n,\ \sigma_i(1)=\binom ni \rangle 
\end{equation} will be called an {\em $n$--Cayley--Hamilton  algebra} or {\em $n$--CH algebra}.  
\end{definition} In other words an $n$--Cayley--Hamilton  algebra $R$ is a quotient, as $\Sigma$--algebra, of one free algebra $F_{\Sigma,n}\langle X \rangle$.

For $n=1$  a  $1$--CH algebra is just a  commutative algebra in which the norm is the identity map or $\sigma_i(a)=0,\ \forall i>1$ and $\sigma_1(a)=a$. Therefore the Theory of 
$n$--Cayley--Hamilton algebras may be viewed as a generalization of Commutative algebra.\begin{remark}\label{PIn0}
For  an $n$--Cayley--Hamilton algebra $R$ the map $\sigma_n$ is a norm and, from Remark \ref{sden} it follows that the $\sigma$ algebra  $\sigma(R)$  coincides with its   {\em norm algebra}.\end{remark} \begin{remark}\label{PIn}
An   $n$--Cayley--Hamilton algebra $R$ satisfies all the polynomial identities of $M_n(\mathbb Z)$.
\end{remark}    

\subsubsection{The main Theorem of Cayley--Hamilton algebras}
 Let   $R$  be any  $n$--Cayley--Hamilton algebra over  a  commutative ring $F $. 
 
 Choosing a set of generators $X$ for $R$  we may present $R$ as a quotient of a free algebra $F_{\Sigma,n}\langle X\rangle=(M_n(F)\otimes A_{X,n})^{PGL(n,-)}$ modulo a $\Sigma$--ideal $I$.
  Now $(M_n(F)\otimes A_{X,n}) I  (M_n(F)\otimes A_{X,n})$ is an ideal of $   M_n(F)\otimes A_{X,n} =M_n(  A_{X,n}) $ so there is an ideal  $J\subset A_{X,n}$, which is   $PGL(n,-)$ stable  with 
 $$ (M_n(F)\otimes A_{X,n})I  (M_n(F)\otimes A_{X,n})=M_n(J)$$   from which  one has a commutative diagram:
 \begin{theorem}\label{immu}
We have a commutative diagram in which the first horizontal arrow $i$ is an isomorphism. The second horizontal arrows are both injective and the vertical maps surjective:
 \begin{equation}\label{lclcl}
\!\!\!\!\!\!\!\!\!\!\!\!\!\!\!\!\!\!\begin{CD}
F_{\Sigma,n}\langle X\rangle @>{i}>\cong>
 M_n[A_{X,n}]^{PGL(n,-)}@>>>M_n[A_{X,n}] \\
@VVV@VVV@VVV \\
R@>{i_R}> >
 M_n[A_{X,n}/J]^{PGL(n,-)}@>>>M_n[A_{X,n}/J] 
\end{CD}
\end{equation}         
If $F\supset \mathbb Q$  then $i_R$ is an {\bf isomorphism}, [Strong embedding Theorem].
\end{theorem}   The fact that in characteristic 0, $i_R$ is an isomorphism depends upon the fact that $GL(n)$,  in characteristic 0, is linearly reductive, and then the proof, see \cite{P5} or \cite{agpr} Theorem 14.2.1, of this Theorem  is based on the so called Reynold's identities.

In general the nature of $i_R$ is not known, we do not know if parts of this statement are true.  This is in my opinion the main open problem of the Theory.

\section{Prime and  simple Cayley Hamilton algebras\label{leCH}}

\subsection{General facts}
\subsubsection{The kernel and the radical}
 
\begin{definition}\label{idita}
\begin{enumerate}\item a  {\em simple} $\Sigma$ algebra is one with no proper $\sigma$ ideals, 
\item a  {\em prime} $\Sigma$ algebra is one in which if $I,J$ are two $\sigma$ ideals with $IJ=0$ then either $I=0$ or $J=0$.

\item Finally a   {\em semiprime} $\Sigma$ algebra is one in which if $I $  is an ideal  with $I^2=0$ then  $I=0$.
\end{enumerate}

\end{definition} Notice that prime implies semiprime.

\begin{definition}\label{ker}
1)\quad  Given a $\Sigma$--algebra $R$ the set
\begin{equation}\label{Ker}
K_R:=\{x\in R\mid \sigma_i(xy)=0,\ \forall y\in  R,\ \forall i\}
\end{equation} will be called the {\em kernel} of the $\Sigma$--algebra.

$R$ is called {\em nondegenerate} if $K_R=0$.\bigskip

2)\quad  The set
\begin{equation}\label{KeR}
\widetilde K_R:=\{x\in R\mid \sigma_i(x )\ \text{is nilpotent},\ \forall y\in  R,\ \forall i \}
\end{equation} will be called the {\em radical} of the $\Sigma$--algebra.

$R$ is called {\em regular} if $\widetilde K_R=0$.\medskip

 By Amitsur's Formula \eqref{comAm}   both $K_R$ and $\widetilde K_R $ are  $\sigma$--ideals of $R$.
 
 Moreover if $B$ is a commutative $A$ algebra   ${}_B  K_R\subset   K_{ {}_BR},\  {}_B\widetilde K_R\subset \widetilde K_{ {}_BR}.$\bigskip

If $I$ is a $\sigma$--ideal  in a $\Sigma$--algebra $R$ we set $K(I)\supset I$ (resp. $\tilde K(I)$)  to be the ideal such that $R/K(I)=K_{R/I}$ (resp. $R/\tilde K(I)=\widetilde K_{R/I}$).

We call $K(I)$  the {\em radical kernel} of $I$ and $\tilde K(I)$  the {\em nil kernel} of $I$, both $\sigma$--ideals.

\end{definition}
 \begin{lemma}\label{CHC} Let $R$ be a $n$--CH  algebra.
 An element  $r\in R$ is nilpotent if and only if all the $\sigma_i(r)$ are nilpotent.
\end{lemma}
 \begin{proof} In one direction  every element  $r\in R$  satisfies its  characteristic polynomial, if $\sigma_i(r)$ is nilpotent for all $i$ we have $r^n$ is a linear combination of the commuting nilpotent elements $\sigma_i(r)r^{n-i}$    hence the claim.  
  Assume now $r$ nilpotent.
The  elements $\sigma_i(r^k), i\cdot k\leq N$ satisfy the   relations of  the corresponding symmetric functions  $e_i(X^k):=e_i(x_1^k,x_2 ^k,\ldots, x_N ^k)$. 

Now, for each $k$  the polynomial ring  $\mathbb Z[x_1 ,x_2  ,\ldots, x_N ]$  is integral over the subring $\mathbb Z[e_1(X^k)  ,e_2(X^k)   ,\ldots, e_N(X^k) ]$ so  $\sigma_i(r)$  satisfies a monic polynomial  of some degree $\ell$  whose coefficients are polynomials in the elements $\sigma_j(r^k)$ (and  with 0 constant coefficient). If $r^k=0$ these coefficients are all 0 so $\sigma_i(r)^\ell=0.$\end{proof}
\begin{proposition}\label{nonde}\begin{enumerate}\item $K_R$ is the maximal $\sigma$--ideal $J$ where $\sigma(J)=0$. 

 \item If  $R$ is  an $n$--CH  algebra we have $ \widetilde K_R  $ is the maximal    ideal $I$ with the property that ${}_BI$ is nil for all $B$.

\item If  $R$ is  an $n$--CH  algebra $R/\widetilde K_R$ is regular., i.e  $\widetilde K_{R/\widetilde K_R}=0$. 

\item If $\sigma_i(a)$  is nilpotent, then $\sigma_i(a)\in \widetilde K_R.$
\item If  $R$ is  an $n$--CH  algebra, over some commutative ring $A$, and $I$ is a nil ideal of $R$  then ${}_BI$ is   is a nil ideal of ${}_BR$ for every commutative $B$ algebra. \end{enumerate}

 \end{proposition}
\begin{proof}
1)\quad The first part is clear.  2)\quad   Follows from the previous Lemma  \ref{CHC}.

3)\quad   If the class of $r\in R$ is in  the radical $\tilde S$ of $ S:=R/\widetilde K_R$ we have that for each $ y\in _BR$,  $\sigma_i(ry)$ is nilpotent, in $_BS$  hence $\sigma_i(ry)$ is nilpotent also in $_BR$  and so  $r\in \widetilde K_R$.

4)\quad   As for the last statement,    $\sigma_j(\sigma_i(a) y)=\sigma_i(a)^j\sigma_j( y)$ is nilpotent.

5)\quad   This follows from  the previous Lemma  \ref{CHC} and Amitsur's Formula.\end{proof}

\begin{corollary}\label{CHC1} Let $R$ be a $n$--CH  algebra with $\sigma$--algebra reduced (no nonzero nilpotent elements) 
  then if $r\in R$ is nilpotent  we have $r^n=0$.
\end{corollary}
In particular we have 
\begin{corollary}\label{CHC2}   1)\quad An  $n$--CH  algebra $R$ is semiprime if and only if  its   $\sigma$--algebra is reduced  and the radical $\widetilde K_R=0$.

2)\quad An  $n$--CH  algebra $R$ is   prime if and only if   its   $\sigma$--algebra is a domain and $\widetilde K_R=0$.

3)\quad An  $n$--CH  algebra $R$ is   simple if and only if   its   $\sigma$--algebra is a field and $\widetilde K_R=0$.
\end{corollary}
\begin{proof} 1)\ Assume $R$ semiprime. If the $\sigma$--algebra contains a non zero nilpotent element $a$ then  $Ra$ is a nilpotent ideal a contradiction.  Since $R$ is a PI algebra, it is semiprime if and only if it does not contain a  nonzero nil ideal. So  since $\widetilde K_R $  is nil  and $R$ is   semiprime $\widetilde K_R = 0$.\smallskip

Conversely if $R$  has an ideal $I\neq 0$  with $I^2=0$  then for each $a\in {}_BI$ we have $\sigma_i(a)$  is nilpotent for all $i$, by Lemma \ref{CHC}  then  $I\subset \widetilde K_R$.\smallskip

2)\ As for the second statement let  us show that the given conditions imply $R$  prime. In fact  given two $\sigma$--ideals  $I,J$ with $IJ=0$ since  $\sigma(I)\subset I,\ \sigma(J)\subset J$ we have $\sigma(I)\sigma(J)=0$. Since these are ideals and $\sigma(R)$ is a domain one of them must be 0.  If $\sigma(I)=0$ then $I\subset   K_R=\{0\}$  since $R$ is semiprime and by the previous statement.

Conversely if $R$ is prime in particular it is semiprime so we must have $\widetilde K_R=0$. If we had two non zero elements  $a,b\in \sigma(R)$ with $ab=0$  we would have $Ra\cdot Rb=0$  and $Ra, \ Rb$ are $ \sigma$ ideals, a contradiction.

3)  If $R$ is simple it is prime so $\sigma(R)$ is a domain. We need to show that if $a\in\sigma(R)$  then $a$ is invertible, and the element $b$ with $ab=1$ is in $\sigma(R)$.  First the ideal $aR$  is $\sigma$ stable so it must be $R$ and  there is an element $b$ with $ab=ba=1$.   Thus the field of fractions $K$ of $\sigma(R)$ is  contained in $R$, but since $K$  is integral over $\sigma(R)$, by the going up Theorem it coincides with $\sigma(R)$.

Conversely if $\sigma(R)$ is a field and $I$ is a nonzero proper $\sigma $ ideal, for every $a\in I$ we must have $ \sigma_i(a)\in I\implies   \sigma_i(a)=0,\ \forall i$. Then $I\subset K_R=\{0\}$.

\end{proof}\begin{proposition}\label{mm} 1)\quad If $R$ is a semiprime $n$ Cayley-Hamilton algebra with   $a\in \sigma(R)$   not a zero divisor in $\sigma(R)$ then $a$ is not a zero divisor in $R$.
\smallskip

2)\quad  If $R$ is a  prime  $\Sigma$--algebra,   $\sigma(R)$ is a domain and $R$ is torsion free relative to $\sigma(R)$.
\end{proposition}
\begin{proof}
1) Let $J:=\{r\in R\mid ar=0\}$,  then $J$ is an ideal and we claim it is nil hence by hypothesis 0.  In fact 
taking one of the functions $\sigma_i$ we have   $0=\sigma_i(ar)=a^i\sigma_i(r)$  implies $\sigma_i(r)=0$  for all $r\in J$, since $r$ satisfies its CH  it must be $r^n=0$.

2) If $a\in \sigma(R)$ and $J:=\{r\in R\mid ar=0\}$ then both $J$  and $Ra$  are   ideals closed under the $\Sigma$  operations  and $J Ra=0$.  Since $R$ is prime and $Ra\neq 0$ it follows that $J=0$.
\end{proof}
 Finally the local finiteness property:
\begin{proposition}\label{finl}
An  $n$--CH  algebra $R$ finitely generated over its $\sigma$--algebra $\sigma(R)$  is a finite $\sigma(R)$ module.
\end{proposition}
\begin{proof}
The Cayley Hamilton identity implies that each element of $R$ is integral over $\sigma(R)$ of degree $\leq n$ then this is a standard result in PI rings  consequence of Shirshov's Lemma, \cite{agpr} 
Theorem 8.2.1..
\end{proof}

\subsubsection{Semisimple algebras} In this section $F$ denotes  an infinite field, this hypothesis could be removed but it simplifies the treatment.
We want to study general CH algebras over  $F$ such that the values of the norm and hence of all the $\sigma_i$  are in  $F$.

First a simple fact. Let $R=R_1\oplus R_2$ be an $F$ algebra and $N: R\to F$ a multiplicative polynomial map. Then setting $e_1,e_2$ the two unit elements of $R_1,R_2$ we have $N(a,b)=N((a,e_2) (e_1,b))=N(a,e_2)N(e_1,b)$. 

Set $N_1(a):=N(a,e_2),\ N_2(b):=N(e_1,b)$. Clearly $N_i: R_i\to F,\ i=1,2$ are both multiplicative polynomial maps and $N(a,b)= N_1(a )N_2( b)$.

\begin{proposition}\label{hh} If $N$ is homogeneous of degree $n$ there are two positive integers $h_1,h_2>0$ with $h_1+h_2=n$ and $N_1(a) $ is  homogeneous of degree $h_1$ while $N_2(b)$ is    homogeneous of degree $h_2$.
\end{proposition} 
\begin{proof}
Restrict $N$ to $F\oplus F$  then  $N$ factors  through a homomorphism of $[(F\oplus F)^{\otimes n}]^{S_n}\to F$.

Now  $[(F\oplus F)^{\otimes n}]^{S_n} $ is the direct sum   of  $n+1$  copies of $F$, each   with unit element the symmetrization $e_1^{h }e_2^{n-h }$  of $e_1^{\otimes h }\otimes e_2^{\otimes n-h }$.   

A homomorphism    $[(F\oplus F)^{\otimes n}]^{S_n}\to F$ thus factors though the projection to one of this summands.       The claim then follows from the remark   that 
$(\alpha e_1+\beta e_2)^{\otimes n}=\sum_{h=0}^n \alpha^h\beta^{n-h}  e_1^{h }e_2^{n-h }$.
\end{proof}\begin{corollary}\label{dsch1}Under the previous hypotheses, $N:R\to F$  a multiplicative polynomial map of degree $n$. If $R=\oplus_{i=1}^kR_i$ then $k\leq n$.
\end{corollary}
\begin{proof}
 $N$ is a product of the norms $N_i$ each with some degree $h_i>0$ and $n=\sum_ih_i$.

\end{proof}

\begin{corollary}\label{dsch}Under the previous hypotheses, $R$ is an $n$--CH algebra if and only if  $R_i$ is an $h_i$ CH--algebra for $i=1,2$.

\end{corollary}
\begin{proof}
$$\chi_{(a,b)}(t)= N( (te_1-a,te_2-b)= N_1 (te_1-a  )  N_2(te_2-b)  = \chi^1_{  a }(t) \chi^2_{  b }(t) ,$$   so  if  $R_i$ is an $h_i$ CH--algebra for $i=1,2$ we have  \begin{equation}\label{cadd}
\chi_{(a,b)}((a,b))=(\chi^1_{  a }(a),\chi^1_{  a }(b))(\chi^2_{  b }(a),\chi^2_{  b }(b))=0.
\end{equation} 

Conversely assume  $\chi_{(a,b)}((a,b))=0$.  Given $a\in R_1$ take $\alpha\in F$  so $\chi^2_{ \alpha}(t)=(t-\alpha)^{h_2}$  and
$$\chi_{(a,\alpha)}((a,\alpha))=\chi^1_{  a }(a)(a-\alpha)^{h_2} =0,\ \forall \alpha\in F.$$  The minimal polynomial $f(t)$ of $ a$ thus divides    $\chi^1_{  a }(t)(t-\alpha)^{h_2}  ,\ \forall \alpha\in F $   hence it divides    $\chi^1_{  a }(t)$ so $\chi^1_{  a }(a)=0$ and similarly for $b$.

\end{proof}

 A semisimple algebra  $S$ finite dimensional over a field $F$  is isomorphic to the direct sum  $S=\oplus_iM_{k_i}(D_i)$  of matrix algebras over division rings  which are finite dimensional over $F$. 
 
  We treated the theory in characteristic 0 in \cite{ppr}, so we assume that $F$ has some positive characteristic $p>0$.  Let $G_i$ denote the center of $D_i$.
 
 If $F$  is separably closed then all   the $D_i=G_i$  are fields, purely inseparable over $F$ (this follows from the fact that a division ring $D$ of degree $n$ has a separable element of degree $n$ over the center, see for instance Saltmen \cite{salt}). We ask in general which norms exist  on $S$ with values in $F$ which make $S$ an $n$--Cayley--Hamilton algebra.\smallskip
 
 We start with a special  case.
   
  Given   two lists $\underline m:=m_1,\ldots,m_k$ and $\underline a:=a_1,\ldots,a_k$ of positive integers with $\sum_jm_ja_j=n$  consider the algebra  with norm $N$\begin{equation}\label{Fma}
F(\underline m ;\underline a ):=\oplus_{i=1}^k M_{m_i}(F),\ \subset M_n(F),\ N(  r_1,\ldots,r_k)=\prod_{i=1}^k \det(r_i)^{a_i}\end{equation} where the $i^{th}$ block is repeated $a_i$.

$F(\underline m ;\underline a )$  is a subalgebra (of block diagonal matrices) of $M_n(F)$ and then the norm $N$ equals the determinant, hence it is an  $n$ Cayley--Hamilton algebra, and, as $\sigma$-- algebra, it is {\em simple}.

  Conversely  we have the standard:
  \begin{proposition}\label{ssn}
If $F$ is algebraically closed and $S\subset M_n(F)$ is a semisimple algebra  then it is one of  the algebras $F(m_1,\ldots,m_k;a_1,\ldots,a_k)$.
\end{proposition}
\begin{proof} 
A semisimple algebra $S$  over $F$ is of the form $S=\oplus_{i=1}^k M_{m_i}(F)$. 

An embedding of $S$  in $M_n(F)$ is a faithful $n$--dimensional representation of $S$. Now the  representations of $S$  are direct sums  of the irreducible representations $ F^{m_i} $ of the blocks  $M_{m_i}(F)$, and a faithful $n$--dimensional representation of $S$ is thus of the form
$$\oplus_i (F^{m_i})^{\oplus a_i},\ a_i\in\mathbb N,\ a_i>0,\ \sum_i a_im_i=n. $$
For this representation the algebra $S$  appears as block diagonal matrices, with an $m_i\times m_i$ block  repeated $a_i$ times. The norm is  the determinant described by Formula \eqref{Fma}.

\end{proof}Formula \eqref{Fma} can be made axiomatic.  Assume $F$ is an infinite field.
\begin{theorem}\label{ssac}
Suppose that $F(\underline m   ):=\oplus_{i=1}^q M_{m_i}(F)$ is equipped with a Norm $N$ of degree $n$ with values in $F$. Then there are positive integers $a_i$ with $\sum_ia_im_i=n$  so that
 \begin{equation}\label{pono}
N(r_1,\ldots,r_q)=\prod_{i=1}^q\det(r_i)^{a_i}. 
\end{equation}
Hence  it   is the $n$ Cayley--Hamilton algebra $F(\underline m ;\underline a )$.
\end{theorem}
\begin{proof} From Corollary \ref{dsch} we are reduced to the case $R=M_h(F)$. In this case the statement is a special case of Theorem \ref{nazz}.
\end{proof}

 \subsubsection{An abstract Theorem}
We have an even more abstract Theorem:
\begin{theorem}\label{simc}
Let $F$ be an algebraically closed field and $S$   an $n$ Cayley--Hamilton algebra with norm in $F$ and radical $\widetilde K_S$. 

Then $S/\widetilde K_S$ is finite dimensional, simple,  and isomorphic to one of the algebras  $F(m_1,\ldots,m_k;a_1,\ldots,a_k)$ as algebra with norm.
\end{theorem} 
\begin{proof} First remark that,  since the values of $\sigma_i$ are all in $F$ we have that the kernel equals the radical $K_S$.                  

Passing to $S / K_S$ we may thus assume that $  K_S=0$. 
Let us first assume that $S$ is finite dimensional, then by Proposition \ref{nonde}   we   have that $S$ is a semisimple algebra so it is of the form $S=\oplus_{i=1}^k M_{m_i}(F)$. Since it is an $n$ Cayley--Hamilton algebra the statement follows from Theorem \ref{ssac}.

Now let us show that it is finite dimensional. For any choice of  a finite set of elements  $A=\{a_1,\ldots,a_k\}\subset S$ let $S_A$ be the subalgebra generated by these elements, since each $\sigma$ takes values in $F$ this is also a $\Sigma$--subalgebra. By  a standard theorem of PI theory since $S$ is algebraic of bounded degree each $S_A$ is finite dimensional. Then if $J_A$ is the radical of $S_A$  we have by the previous part that $\dim S_A/J_A\leq n^2$.  Let us  choose $A$ so that  $\dim S_A/J_A$ is maximal. We claim that $S=S_A$ and $J_A=0$.

First let us show that $J_A\subset K_S$ the Kernel  of $S$. Let  $a\in J_A$  and $r\in S$; we need to show that $ \sigma_i(ra)=0,\ \forall i$. If $r\in S_A$ this is the previous statement, if $r\notin S_A$ then $S_{A,r}\supsetneq S_A$ and we claim that $J_{A,r}\supset  J_A$, in fact   otherwise $\dim S_{A,r}/J_{A,r}> \dim S_A/J_A$ a contradiction.

Then by the previous argument $\sigma_i(ar)=0$ so $a\in K_S$ but since $S$ is simple $K_S=0$ and $J_A=0$. Next if $S_A\neq S$  we have again some $S_{A,r}\supsetneq S_A$ and now $J_{A,r}\neq 0$ a contradiction.
\end{proof}

\subsection{General semisimple algebras}
If $R$  is a simple PI algebra,  over a field $F$, we have $R=M_k(D)$  with $D$  a division ring finite dimensional  over its center $G\supset F$ (Theorem 11.2.1 of \cite{agpr}), let  $\dim_GD=h^2$. If furthermore $R$   is finite dimensional over  $F$  let  $ \dim_FG=\ell$.

In this last case the algebra $R$  is endowed with a canonical {\em Norm} homogeneous of degree $kh\ell$ which is a composition of two norms
$$N_{R/F}= N_{R/G}\circ N_{G/F}$$ The   norm $N_{R/G}$ can be defined  as follows. We take a maximal subfield $M \subset D$ separable over $G$  then: $$M_k(D)\otimes_GM=M_{k\cdot h }(M).$$  If $a\in M_k(D)$  define as Norm   $N_{R/G}(a):=\det(a\otimes 1) $ as matrix.    

It is a standard  fact that  $N_{R/G}(a)\in G$.  As for $N_{G/F}(g),\ g\in G$  one takes the determinant of the multiplication by $g$  a $\ell\times\ell$ matrix over $F$.  

Assume   $G$ is separable over $F$  and $\bar F\supset G$  is the separable closure of $G$ and $F$. We have the $\ell$,  $F$--embeddings of $G$  in  $\bar F$,  $\gamma_1,\ldots,\gamma_\ell$  given by Galois Theory:
 \begin{equation}\label{ga}
G\otimes_F\bar F= \bar F^\ell,\  g\otimes 1=(\gamma_1(g),\ldots,\gamma_\ell(g))\implies N(g)=\prod_i\gamma_i(g)\in F.
\end{equation}

Notice that, in this case we also have a trace $tr(a)\in F,\ \forall a\in R$, the separability condition is given by the fact that the trace form $tr(ab)$ is non degenerate.\smallskip

In general let $L$ be the separable closure  of $F$ in $G$, and $a=[L:F],\ p^k=[G:L]$.  Let $\bar F$ be a separable closure of $F$   we still have that 
$$L\otimes_F\bar F=\bar F^{\oplus a},\  D\otimes_F\bar F= D\otimes_L(L\otimes_F\bar F)=\oplus_i D\otimes_L  \bar F$$
where in the summands  $L$ embeds in $\bar F$   by the $a$ different embeddings given by Galois Theory as in Formula \eqref{ga}.

Moreover   $G\otimes_L\bar F$ is a field by Theorem  9 page 163 of \cite{Jac}. Finally 
\begin{lemma}\label{mati}
$D\otimes_L  \bar F=D\otimes_G(G\otimes_L  \bar F)\simeq M_h(G\otimes_L\bar F)$ and $$M_k(D)\otimes_F\bar F=M_k(D)\otimes_G(G\otimes_F\bar F)=M_k(D)\otimes_G(G\otimes_L(L\otimes_F\bar F))$$\begin{equation}\label{spli}
=\oplus_i M_{hk}(G\otimes_L\bar F).
\end{equation}
\end{lemma}
\begin{proof}
$D$  contains a maximal subfield $M$ separable over $G$  so generated by a single element $a$ satisfying an irreducible separable polynomial  $f(x)=x^h+\sum_{i=1}^h\alpha_ix^{h-i},\ f(a)=0$   with coefficients $\alpha_i$ in $G$.  We have for some  power $a^{p^k}$  that $a^{p^k}$  satisfies  $  x^h+\sum_{i=1}^h\alpha_i^{p^k}x^{h-i}$  with coefficients $\alpha_i^{p^k}$ in $ \bar F$.  Since $M$ is purely inseparable over $G[a^{p^k}]$ we have that $a^{p^k}$ is also a generator of $M$  but being  separable over $\bar F$  it is in $\bar F$. Therefore $G\otimes_L\bar F$ is a splitting field for $D$.
\end{proof} As for the norm  $N_{M_k(D)/F}$ one has to embed $G\otimes_L\bar F\subset M_\ell(\bar F)$ and then in the embedding
 $$M_k(D)\otimes_F\bar F\subset \oplus_i M_{hk p^k}( \bar F)\subset M_{hk p^k\ell}( \bar F),\ a\mapsto (a_1,\ldots,a_\ell)$$ we have
 $$ N_{M_k(D)/F}(a )=\prod_i\det(a_i).
 $$\begin{proposition}\label{redt}
Under this norm  $M_k(D)$ is a $k\cdot h\cdot p^k\cdot\ell$ CH algebra.
\end{proposition}
\begin{proof}
The norm is induced by the determinant of the previous Formula.
\end{proof} 

We ask now what is the general form of a norm and we will see that, contrary to  what happens in characteristic 0 in general there are norms which have degree strictly less than that of the canonical norm.

 We now do not even assume that $D$ is finite dimensional over $F$.  
 
 Let again $L$  be the separable closure of $F$ in $G$,
 and  consider    a norm  $N:R=M_k(D)\to F$  that is a multiplicative polynomial map homogeneous of some degree $n$.
 \begin{theorem}\label{maps}     $[L:F]=\ell<\infty$. There is  a   minimum integer $k$ such that  $g^{p^k}\in L,\ \forall g\in G$.
There is a $b\in\mathbb N$ such that $n= khbp^k\ell $. The norm $N$ depends only upon $b$ and will be denoted by $N_b$ and maps as
\begin{equation}\label{fat}
\begin{CD}
N_b:M_k(D)@>N_{R/G}>>G@>{a\mapsto a^{bp^k}}>>L@>N_{L/F}>>F .
\end{CD}
\end{equation}
\end{theorem}
That is $N_b=N_1^b$.
The proof is in several steps. Denote again by $\bar F$ a separable closure of $F$.  Restricting the norm to $L\otimes_F\bar F$ we have, bu Corollary \ref{dsch1},  that $[L:F]=\ell$ is finite.  The $\ell$ divides $n$ by a combination of Formula \eqref{ga} and Theorem \ref{ssac}.  Lemma \ref{mati} still holds. 
 
 The norm $N$ induces a norm $N:M_k(D)\otimes_F\bar F\to \bar F$ which, by Proposition \ref{hh} is the product of the norms $N_i$ in the $\ell$ summands of  Formula  \eqref{spli}. If $\gamma$ is an automorphism of $\bar F$ over $F$  we have  $N\circ 1\otimes\gamma=\gamma\circ N$.  This allows us to say that the norms $N_i$ induced according to Proposition \ref{hh} in the  $a$ summand of Formula  \eqref{spli}  have all the same degree $m$  and $n=ma$.

 Set  $ m =hk$  we   have to analyze  the  norms  $N:       M_{m}(G\otimes_L\bar F)\to \bar F$ which by abuse of notation we still think of degree $n$.
 
 The first case is when $m=1$.

Changing notations let $G\supset F$    be   purely inseparable  over $F$. \begin{lemma}\label{rra}
The map   $\pi_n:G^{\otimes n}\to G, \ a_1\otimes a_2\otimes \cdots \otimes a_n\mapsto \prod_ia_i$  is a surjective homomorphism with kernel the nil radical of   $G^{\otimes n}$.
\end{lemma}
\begin{proof}
By induction on $n$.   The map $\pi:G\otimes G\to G,\  \pi(a\otimes b)=ab$  is a homomorphism with kernel the  ideal  
 generated by $a\otimes 1-1\otimes a$ which is nilpotent sinc, for some $h$ we have $a^{p^h}\in F$ so  $(a\otimes 1-1\otimes a)^{p^h}= a^{p^h}\otimes 1-1\otimes a^{p^h}=0.$  Then  by induction $\pi_n$  must factor through  $J\otimes G$ with $J$ the radical of  $G^{\otimes n-1}$.
 
 But $G^{\otimes n}/J\otimes G=G\otimes G$ and we are at the beginning of the induction.
\end{proof}
\begin{lemma}\label{pui}
If $N:G\to F$ is a multiplicative polynomial map of degree $n$ then   there is a minimal $k $ so that $a^{p^k}\in F,\ \forall a\in G$,
$n=bp^k,\ b\in\mathbb N$ and    $N(a)=a^{n}.$  
\end{lemma}
\begin{proof}
Let $a\in G$ be such that $a^{p^h}\in F$ with $h$ minimal,  restrict $N$ first to $G=F[a]$.  
A multiplicative polynomial map then factors though
\begin{equation}\label{mop}
N:G\stackrel{u\mapsto u^{\otimes n}}\longrightarrow  (G\otimes_FG\otimes \cdots\otimes_FG)^{S_n}\stackrel{\bar N}\longrightarrow F
\end{equation}
with $\bar N$ a homomorphism which thus  vanishes on the radical.

Now the map   $\pi:G^{\otimes n}\to G, \ a_1\otimes a_2\otimes \cdots \otimes a_n\mapsto \prod_ia_i$  is a surjective homomorphism with kernel the radical of   $G^{\otimes n}$. It follows that $\pi$ restricted to  $(G\otimes_FG\otimes \cdots\otimes_FG)^{S_n}$  induces an isomorphism of  $(G\otimes_FG\otimes \cdots\otimes_FG)^{S_n}$  modulo its radical  with  a field  $L$ with $F\subset L\subset G$. Thus if $\bar N$   exists, since it factors through $(G\otimes_FG\otimes \cdots\otimes_FG)^{S_n}$  we must have $L=F$.  

Next $(G\otimes_FG\otimes \cdots\otimes_FG)^{S_n}$ is generated by the elementary symmetric functions in the elements $1^{\otimes i}\otimes a\otimes 1^{\otimes n-i-1}$ so the image of $\pi$  is generated by the elements $\binom n j a^i$. In particular if they have to lie in $F$  we must have that $a^n\in F$ so $n$ is a multiple of  the minimal  $p^h$ for which $a^{p^h}\in F$.

 For general $G$  the statement follows by restricting to all subfields $F[a]$ of previous type, assuming the existence of $N$  we must in particular have that for each $a\in G$  we have $[F[a],F]= p^d$ with $p^d$ dividing $n$. Therefore there is a minimum $k $ so that $a^{p^k}\in F,\ \forall a\in G$ and  $n$ is a multiple of  $p^k$.
 
 Finally $N$ is  the canonical norm  $G\to (G\otimes_FG\otimes \cdots\otimes_FG)^{S_n}=F$ which is $a\mapsto a^n$.

\end{proof}  
\begin{remark}\label{allx}
If  $\tilde F$ is an algebraic closure of $F$,  setting  $y=x-a$$$ \tilde F\otimes_FF[a]=\tilde F[x]/(x^{p^k}-a^{p^k})\simeq \tilde F[y]/y^{p^k}.$$ The norm $N$  extends to
\begin{equation}\label{mop1}
N:\tilde F[y]/y^{p^k}\stackrel{a\mapsto a^{\otimes n}}\longrightarrow  [(\tilde F[y]/y^{p^k})^{\otimes n} ]^{S_n}\stackrel{\bar N}\longrightarrow \tilde F.
\end{equation} There is a unique $\bar N$  in Formula  \eqref{mop1} which vanishes on the radical.
\end{remark}  
{\em We pass to general $m$.} \quad
Let  $N:M_m(G)\to F$ be a multiplicative polynomial map, over $F$, of degree $n$.    Consider the subspace of diagonal matrices  $G^m\subset M_m(G)$, the norm $N$ restricted to 
 $G^m$ is a product of norms  $N_i$  for the various summands, but the symmetric group  $S_m\subset M_m(G)$   permutes the summands so the norms are all the same   and  by Lemma \ref{pui} there is  $b$ so that $N(a_1,\ldots,a_m)=\prod_ia_i^{bp^k}=\det((a_1,\ldots,a_m))^{bp^k}$. 
 
 Next consider elementary matrices  $e_{i,j}(\alpha):=1+\alpha e_{i,j},\ i\neq j,\ \alpha\in G$. 
 
 We have  $e_{i,j}(\alpha) e_{i,j}(\beta)= e_{i,j}(\alpha+\beta)$ so $N$ on this subgroup is a  multiplicative polynomial map  from   $G$, as additive $F$ vector space, to the multiplicative $F$. In a basis of $G$ over $F$ it is thus a polynomial of some degree $k$  with
 $$f(x_1+y_1,\ldots,x_h+y_h)= f(x_1 ,\ldots,x_h )f( y_1,\ldots, y_h)$$  this by degree implies that $f=1$. Then by the usual Gaussian  elimination any matrix $A$  is a product of a diagonal matrix times elementary matrices so that  hence $N(A)=\det(A)^{bp^k}$ for all $A$. 
 
 Finally   the norm $N_b$ of Formula \ref{fat}  becomes under tensor product  with $\bar F$ the one of Theorem \ref{maps} so  the Theorem follows.\qed

Let us summarize what we proved. Consider a  general semisimple algebra over $F$
 $R=\oplus_{i=1}^pM_{k_i}(D_i), \dim_{G_i}D_i=h_i^2$ where $G_i$ is the center of the division algebra $D_i$.  Let $L_i\subset G_i$ be the separable closure of $F$,  $ \dim_FL=\ell_i$ and $k_i$  the minimum integer such that  $a^{p^{k_i}}\in L_i,\ \forall a\in G_i$.  Given positive integers  $b_i,\ i=1,\ldots, p$  we may define the norm
\begin{equation}\label{trconp}
N(a_1,\ldots,a_p):=\prod_{i=1}^pN_{b_i}(a_i), \ a_i\in M_{k_i}(D_i),\ N_{b_i}(a_i)\quad\text{Formula }\eqref{fat}.
\end{equation}
\begin{theorem}\label{trach} The algebra $R=\oplus_{i=1}^pM_{k_i}(D_i)$ with the previous norm is an $n$  CH algebra with $n=\sum_ik_ih_ib_ip^{k_i}\ell_i$.

Conversely any  norm on $R$    is of the previous form.

\end{theorem}
\begin{proof}
 From Corollary \ref{dsch}  we may reduce to $R=M_{k }(D )$ and then, by splitting to  the case in which $G$ is purely inseparable over $F$, $R=M_{\ell }(G )$ and $N:M_\ell(G)\to F,$ $ N(a)=\det(a)^{bp^k}$.  The characteristic polynomial is thus  $\chi(t)=\det(t-a)^{bp^k}$.  Then clearly this vanishes for $t=a$ by the usual CH Theorem for $M_\ell(G)$. \end{proof}
\begin{corollary}\label{enq}
Let  $R=\oplus_{i=1}^pM_{k_i}(D_i)$  be an $n$  CH algebra over a field $F$ as in Theorem \ref{trach}. If for one $i$  we have  that the dimension of $M_{k_i}(D_i)$ over its center $G_i$  equals $n^2$  then $R=M_{k_i}(D_i), \ F=G_i$ and the norm is the reduced norm.
\end{corollary}
For the next Theorem we need:
\begin{lemma}\label{sss}
Let $S$ be a semiprime algebra with center an infinite field $C$ and each element satisfies an algebraic equation of degree $n$ over $C$. Then $S$ is a finite dimensional central simple algebra.
\end{lemma}
\begin{proof}
$S$    satisfies a polynomial identity so it is enough to prove that $S$ is a prime algebra (cf. \cite{agpr} Theorem 11.2.6).

Let $P$ be a prime ideal of $S$, the prime algebra $S/P$ is algebraic over $C$    so its center is a field and being a PI ring it is a simple algebra, thus isomorphic to $M_k(D)$ with $D$ a division algebra finite dimensional over its center. In particular $P$ is a maximal ideal.  By \cite{agpr} Theorem 1.1.41 we have $\{0\}=\bigcap P$  is the intersection of all prime (maximal ideals). If  there are only finitely many maximal ideals $P_1\cap P_2\cap\ldots\cap P_k=\{0\}$  then $S=\oplus_i S/P_i$  and its center is a field only if $m=1$. But if we have $m>n$  maximal ideals we still have $S/\cap _{i=1}^m P_i=\oplus_{i=1}^m S/P_i$ contains $C^m$  which contains elements which are not algebraic of degree $n$ over $C$.  \end{proof}

\begin{theorem}\label{ssre1}

 If  $S$ is a $\sigma$--simple $n$ Cayley--Hamilton algebra   then $S$ is    isomorphic to one of the algebras  of Theorem \ref{trach}.

\end{theorem} 
 \begin{proof}  By Corollary \ref{CHC2}  the algebra $\sigma(R)=F$ is a field.  Let $C$ be the center of $R$ it is  also an  $n$ Cayley--Hamilton algebra, commutative and with no nilpotent elements, we claim it is a finite direct sum of fields. 
 
 In fact  we claim that in $C$ there are at most $n$ orthogonal idempotents. In fact if $e_1,\ldots,e_m$ are orthogonal idempotents we have that $S=\oplus_i Se_i$, by Corollary \ref{hh}  we have $n=\sum_ih_i, h_i>0$ and this claim follows.
 
 Then if $1= e_1+\cdots+e_m$ is a decomposition into primitive  orthogonal idempotents we have $C=\oplus_i C_i,\ C_i:=Ce_i$. Each $C_i$  has no nilpotent elements, no non trivial  idempotents  and it is algebraic over $F$  then it is a field.
 
We claim that $S_i:=Se_i$  is an $h_i$  simple CH algebra. 

In fact  clearly an ideal of $S_i$ is an ideal of $S$ so there are no nil ideals, the $\sigma$ algebra is contained in $F$ which is a field so, by the argument of Corollary \ref{CHC2}   is also a field and the conclusion of Corollary \ref{CHC2}  applies.

In positive characteristic the  $\sigma$ algebra of $S_i$ need not be $F$ as the following example shows. Let $F$ be of positive characteristic $p$, consider on $R=F\oplus F$  the norm $N(a,b)=a^pb$, the  $\sigma$ algebra of $R$ is $F$ but that  of the first summand is $F^p$  which may be different from $F$.

Now, changing notations we may  assume that the center of $S$ is a field $C\supset F=\sigma(S)$ and conclude by Lemma \ref{sss}.
 
  \end{proof}
  Recall that,   by Proposition \ref{mm} a $\sigma$--prime $n$ Cayley--Hamilton algebra $S$ is torsion free over $\sigma(R)$ which is a domain.  Thus we can embed $S$ in $S\otimes_{\sigma(S)}K $.
  \begin{corollary}\label{stpr}
 If  $S$ is a $\sigma$--prime $n$ Cayley--Hamilton algebra and $K$ is the field of fractions of   $\sigma(S)$    then $S\otimes_{\sigma(S)}K \simeq \oplus_{i=1}^pM_{k_i}(D_i) $  is a $\sigma$--simple $n$ Cayley--Hamilton algebra   isomorphic to one of the algebras  of Theorem \ref{trach} and containing $S$.

There are finitely many minimal prime ideals $P_j=S\cap  \oplus_{i=1,\ i\neq j}^pM_{k_i}(D_i) $ in $S$ with intersection $\{0\}$.
\end{corollary}
\subsection{The Spectrum} 
 In any associative algebra $R$  one can define the {\em spectrum}  of $R$ as the set of all its prime ideals, it is equipped with the {\em Zariski topology}.  
 
 For commutative algebras the spectrum is a contravariant functor with $f:A\to B$ giving $P\mapsto f^{-1}(P)$. But in general a subalgebra of a prime algebra need not be prime  and the functoriality fails.

For $\Sigma$--algebras   $R$ we may define:
$$Spec_\sigma (R):=\{P\mid P \text{ is a prime $\sigma$--ideal}\}.$$  
 For  an $n$ Cayley-Hamilton algebra $R$ we have, by Corollary \ref{CHC2} the map  $j: Spec_\sigma (R)\to Spec (\sigma(R)),\ P\mapsto P\cap \sigma(R) $   and the remarkable fact:
\begin{proposition}\label{iso}
The map  $j: Spec_\sigma (R)\to Spec (\sigma(R)),\ P\mapsto P\cap \sigma(R) $ is a homeomorphism, its inverse  is $\mathfrak p\mapsto \tilde K(\mathfrak pR).$
\end{proposition}
\begin{proof} First we have, for any $\Sigma$--algebra   $R$,  any  $I\subset  \sigma(R)$  an ideal  of $\sigma(R)$ that $IR$ is a $\sigma$--ideal and   $IR\cap \sigma(R)= I$. In fact if $r\in I\cap \sigma(R)$ we have $r=\sum_i a_is_i,\ a_i\in I,\ s_i\in R $ and, by Amitsur's formula  $\sigma_j( r)=\sigma_j(\sum_i a_is_i )$ is a polynomial in elements $\sigma_k(u)$  with $u$ a monomial in the elements  $a_is_i$. So $u=as,\ a\in I$ and $\sigma_k(u)=a^k \sigma_k(s)\in I.$

Let $\mathfrak p\subset \sigma(R)$  be a prime  $\sigma$--ideal.  Since $\sigma(R/ \mathfrak pR )=\sigma(R)/\mathfrak p$ is a domain we have also  $\tilde K(\mathfrak pR)/\mathfrak pR\cap \sigma(R)/\mathfrak p=\{0\}$          hence $\sigma(R/\tilde K(\mathfrak pR))=\sigma(R)/\mathfrak p$.

From  Corollary \ref{CHC2} 2) we have that the ideal $\tilde K(\mathfrak pR),$  is prime. In fact  $\sigma(R/\tilde K(\mathfrak pR))=\sigma(R)/\mathfrak p$ a domain and also the radical of $R/\tilde K(\mathfrak pR)$ is $\{0\}$, by 3) of Proposition \ref{nonde}.
                    
So the composition in one direction is the identity $\mathfrak p=\tilde K(\mathfrak pR)\cap \sigma(R)$.   If $P$ is a prime $\sigma$--ideal  we need to show that $P=\tilde K((P\cap \sigma(R))R)$.

 We certainly have $P\supset\tilde  K((P\cap \sigma(R))R)$ so it is enough to show that, if $P\supset Q$ are two prime $\sigma$--ideals and $P\cap \sigma(R)=Q\cap \sigma(R)$ then $P=Q$. In fact in $R/Q$, a prime $\sigma$--algebra,  we have $\sigma(P/Q)=0$ which implies $P/Q\subset K_{R/Q}=0$.\end{proof}
So  for $n$ Cayley-Hamilton  algebras the spectrum is also a contravariant functor setting $$f:A\to B,\ P\mapsto       \tilde    K(f^{-1}(P)).$$
\paragraph{Some consequences} Assume that $R$ is  an $n$ Cayley-Hamilton  algebra.
Let $\mathfrak p$  be a prime $\sigma$--ideal of $\sigma(R)$  and consider the local algebra $\sigma(R)_{ \mathfrak p}$ and $$R_{ \mathfrak p}:=R\otimes_{\sigma(R)}\sigma(R)_{ \mathfrak p}$$ we have that 
\begin{proposition}\label{locanc}
$R_{ \mathfrak p}$ is   local, as  $\sigma$--ring, with maximal $\sigma$--ideal $\tilde K(R_{ \mathfrak p}\mathfrak p)$. Theorem \ref{trach} gives the possible residue  $\sigma$--simple algebras $R( \mathfrak p):=R_{ \mathfrak p}/\tilde K(R_{ \mathfrak p}\mathfrak p)$. 

\end{proposition}

A special case is when $R( \mathfrak p)$ is simple as algebra and of rank $n^2$ over its center. In this case one can apply Artin's characterization of Azumaya algebras (cf. \cite{agpr} Theorem 10.3.2)  
 and deduce that
\begin{proposition}\label{locanc1}
  If $R( \mathfrak p)$ is simple as algebra and of rank $n^2$ over its center than $R_{ \mathfrak p}$ is a rank $n^2$ Azumaya algebra over its center $\sigma(R)_{ \mathfrak p}$.

\end{proposition} Let us analyze general prime ideals, not necessarily $\sigma$--ideals. 

\begin{proposition}\label{azch}
Let $S$ be  an $n$ Cayley--Hamilton algebra with  $\sigma$--algebra $A$, $P$ an algebra ideal of $S$  which is prime and $\mathfrak p=P\cap A$. Let $F$ be the field of fractions  of  $A/\mathfrak p$.

 Then $S/P\otimes_AF=S/P\otimes_{A/\mathfrak p}F$ is a  simple algebra and $P$ is one of the minimal primes of the prime  $\sigma$--algebra $S/\tilde K(\mathfrak pS)$ (Corollary \ref{stpr}). \end{proposition}
\begin{proof}
 We have $P\supset \tilde K(\mathfrak pS) $ since $\tilde K(\mathfrak pS)$   is nil  modulo $\mathfrak pS$.  
 
 Thus we have a surjective map  $S/\tilde K(\mathfrak pS)\to S/P$  which induces a  surjective map  $S/\tilde K(\mathfrak pS)\otimes_AF\to S/P\otimes_AF$  with $F$ the field of fractions  of  $A/\mathfrak p$.

 Since $S/P\otimes_AF=S/P\otimes_{A/\mathfrak p}F$ is a prime algebra and, Corollary \ref{stpr},  $S/\tilde K(\mathfrak pS)\otimes_AF=\oplus_iS_i$  is a direct sum of simple algebras   we must have   $S/P \otimes_AF=S_i$  for one of the summands and the claim follows.
\end{proof}  This is a strong form of {\em going up} and {\em lying over} of commutative algebra in this general setting.\medskip

\begin{corollary}\label{mama}
Let $S$ be  an $n$ Cayley--Hamilton algebra with  $\sigma$--algebra $A$, $M$ an algebra ideal of $S$  which is maximal and $\mathfrak m=M\cap A$. Then $\mathfrak m$ is a maximal ideal of $A$.
\end{corollary}
\begin{proof}
We have that $S/M$ is a simple algebra  integral over $A/\mathfrak m$ hence its center, a field, is  integral over $A/\mathfrak m$. Thus  $A/\mathfrak m$ is a field from the going up theorem.\end{proof}
\begin{lemma}\label{azch2}
Let $S$ be  an $n$ Cayley--Hamilton algebra with  $\sigma$--algebra $A$, $M$ an algebra ideal of $S$ so that $S/M$ is simple of dimension $n^2$ over its center  and $\mathfrak m=M\cap A$. Then $M=\mathfrak mS$.
\end{lemma}
\begin{proof} The 
algebra       $\bar S:=S/\mathfrak m S$   is an $n$ Cayley--Hamilton algebra with  $\sigma$--algebra the field $F=A/\mathfrak m  A $ by the previous Lemma.  

We have that  $\bar S/\tilde K(\bar S)$  is a simple $\sigma$--algebra to which we can apply
Theorem \ref{trach}.  From Corollary \ref{enq} t  and the fact that one of the simple summands of $R:=\bar S/\tilde K(\bar S)$ is $S/M=M_k(D)$   simple of dimension $n^2$ over its center we have that     $R=M_k(D)$ with center $F$. Since $S$ satisfies the PI's of $n\times n$ matrices, by Artin's characterization of Azumaya algebras we have that $S$  is Azumaya over its center $A$. By Theorem \ref{nazz} the norm takes as values 
the center $A$ which thus must equal $F$ and so  $S=M_k(D)$.
\end{proof}
\begin{theorem}\label{azch1}
Let $S$ be  an $n$ Cayley--Hamilton algebra with $\sigma$--algebra $A$ a local ring with maximal ideal $\mathfrak m$. Let $M$ be an algebra ideal of $S$ so that $S/M$ is simple of dimension $n^2$ over its center, then $M= \mathfrak m S$ and $S$ is a rank $n^2$ Azumaya algebra over $A$.
\end{theorem}

\subsection{$T$--ideals and relatively free algebras}
  One purpose of this paper is to classify and then study prime and semiprime $T$--ideals in a free  $CH_n$--algebra.  In order to explain the meaning and the results of this program we recall first the classical theory of polynomial identities of which the present paper is a natural development.\medskip

 Recall that, given an associative algebra $R$  over an infinite field $F$, a {\em polynomial identity} of $R$  is an element $f(x_1,\ldots,x_\ell )\in F\langle x_1,\ldots,x_\ell ,\ldots\rangle$ of the free algebra in countably many variables which vanishes under all evaluations  in $R$, $x_i\mapsto r_i\in R$  of the variables $x_i$.

The set $I$ of polynomial identities of $R$  is an ideal of $F\langle x_1,\ldots,x_\ell ,\ldots\rangle$  with the special property of being closed under all substitutions  of the variables  $x_i\mapsto g_i\in F\langle x_1,\ldots,x_\ell ,\ldots\rangle$.  Such an ideal is called a $T$--ideal. Conversely any $T$--ideal $I$ is the ideal of polynomial identities of some algebra $R$, and we can take as $R= F\langle x_1,\ldots,x_\ell ,\ldots\rangle/I$.\bigskip

We want to investigate now  the structure of prime $T$--ideals (in the sense of $\Sigma$--algebras) in the free $n$ Cayley Hamilton algebra $F_{\Sigma,n}\langle X\rangle$. For simplicity let us assume that all algebras are over an infinite field, though this is not strictly necessary as the reader may show.

If $P$  is such an ideal then it is the $T$--ideal of $\Sigma$--identities  of the prime $\Sigma$--algebra    $R:=F_{\Sigma,n}\langle X\rangle/P$. These identities are also satisfied by the algebra of fractions of $R$ so that finally a prime $T$--ideal of $\Sigma$--identities is the  $T$--ideal of $\Sigma$--identities  of a simple $\Sigma$--algebra $S$ with  $\sigma(S)=L\supset F$ a field. Replace $S$, with   $\check  S:=S\otimes_L\check  L$  where $\check  L$ is an algebraic closure of $L$.  Although this algebra is not necessarily simple the $\Sigma$--identities with coefficients in $F$ of $S$ and $\check S$, coincide.

If $\bar L\subset \check L$ is the separable closure of $L$ we have that  $\bar S:=S\otimes_L\bar L= \oplus_iM_{m_i}(G_i)$ with $G_i$ purely inseparable over $\bar L$ and furthermore there is a minimum $k_i$  with $g^{p^{k_i}}\in \bar L, \ \forall g\in G_i$.

The structure of $\check  S:=\bar S\otimes_{\bar  L}\check L= \oplus_iM_{m_i}(G_i\otimes_{\bar  L}\check L)$ can be deduced from Remark \ref{allx}. The commutative algebras $G_i\otimes_{\bar  L}\check L$  are obtained from the algebraically closed field $\check L$  by adding nilpotent elements $a_i ,\ i\in I,$ of degrees $p^{h_i},\ h_i\leq k_i$, i.e. $G_i\otimes_{\bar  L}\check L=\check L[a_i],\ i\in I,\  a_i^{p^{h_i}}=0$. The norm $N$ is a product of the norms $N_i:M_{m_i}(\check L[a_i])\to \check L[a_i]$  is given by $A\mapsto \det(A)^{p^{k_i}}.$  Here $A\in M_{m_i}(\check L[a_i])$ is a matrix  with entries  $a_{i,j}+u_{i,j}, \ a_{i,j}\in \check L$ and $u_{i,j}^{p^{k_i}}=0$.  So that $\det(A)=\det((a_{i,j}))+u,\ u  ^{p^{k_i}}=0$ and finally $\det(A)^{p^{k_i}}=\det((a_{i,j})) ^{p^{k_i}}.$

This means that the norm $N_i$ is obtained from the norm $\det(A)^{p^{k_i}}$ of $M_{m_i}(\check L)$  by extending the coefficients, thus  as far as $\Sigma$--identities, the algebra  $\check  S:=  \oplus_iM_{m_i}(\check L[a_i])$ is equivalent to $ \oplus_iM_{m_i}(\check L )$. Finally, since $F$ is an infinite field, this is PI equivalent to $ \oplus_iM_{m_i}(F)$.

The possible norms on $ \oplus_iM_{m_i}(F)$ are given by Theorem \ref{ssac} and are of the form given by Formula  \eqref{pono}, thus depend only on $q$ integers $m_i, a_i$ with $\sum_{i=1}^qm_i  a_i=n$.

 To complete the classification of prime $T$--ideals we have only to show that two algebras associated to two different sequences of  $q$ integers $m_i, a_i$ with $\sum_{i=1}^qm_i  a_i=n$ are not PI equivalent. 
For this it is enough to see that the corresponding relatively free algebras contain the information on these numbers. This part is identical to that developed in the previous paper \cite{Pr8}  to which we  refer.
\begin{theorem}\label{ttid}
Prime $T$--ideals (in the sense of $\Sigma$--algebras) in the free $n$ Cayley Hamilton algebra $F_{\Sigma,n}\langle X\rangle$ are classified by   sequences of  $q$ integers $m_i, a_i$ with $\sum_{i=1}^qm_i  a_i=n$. To such a sequence one associates the $T$--ideal of identities of the algebra of Formula \eqref{Fma}
$F(\underline m ;\underline a )$.
\end{theorem}

\bibliographystyle{amsalpha}

\end{document}